\documentclass[11pt]{amsart}
\usepackage{amsmath}
\usepackage{amssymb}
\usepackage{amsthm}

\usepackage{latexcad}

\newcommand{\mpc}[1]{\marginpar{}}%\hskip .1cm\footnotesize{#1}}}
\newcommand{\npc}[1]{\marginpar{}}%\hskip .1cm\footnotesize{#1}}}

\numberwithin{equation}{section}

\newcounter{AbcT}

\newtheorem {Theorem}    {Theorem}[section]
\newtheorem* {Definition} {Definition}

\newtheorem {Lemma}      [Theorem]    {Lemma}
\newtheorem {Corollary}  [Theorem]    {Corollary}
\newtheorem {Proposition}[Theorem]    {Proposition}
\newtheorem {Claim}      [Theorem]    {Claim}

\newtheorem {Example}    [Theorem]    {Example}

\newtheorem {Conjecture}[Theorem]    {Conjecture}
\newcounter{DM@bibnum}

\newcommand {\F} {{\mathbb F}}

\newcommand {\R} {{\mathbb R}}

\newcommand {\Z} {{\mathbb Z}}

\newcommand{\la}{\langle}
\newcommand{\ra}{\rangle}

\newcommand{\frc}{\displaystyle\frac}

\def\Vert{{\mathcal V}}
\def\Edg{{\mathcal E}}

\def\deg{{\rm deg\,}}

\def\Ker{{\rm Ker\,}}
\def\Im{{\rm Im\,}}

\def\diag{{\rm diag\,}}

\def\NSL_2{{\mathcal N SL_2}}

\def\eps{\varepsilon}

\def\lam{\lambda}            
                
\def\phi{\varphi}

\def\calC{{\mathcal C}}

\def\calI{{\mathcal I}}

             \def\Hgal{{\widetilde H}}

             \def\Tgal{{\widetilde T}}

             \def\Xgal{{\widetilde X}}

\def\ebar{\bar e}

\def\hbar{\bar h}

               \def\Ahat{\widehat {\mathstrut A}}

               \def\Yhat{\widehat {\mathstrut Y}}

\def\grA{{\mathfrak A}}

\def\dbC{{\mathbb C}}

\def\dbF{{\mathbb F}}

\def\dbN{{\mathbb N}}

\def\dbR{{\mathbb R}}

\def\dbZ{{\mathbb Z}}

\def\Fp{{\dbF_p}}

\def\Fq{{\dbF_q}}

\def\epsilon{\varepsilon}
\def\eps{\varepsilon}
\DeclareMathOperator{\GL}{GL}
\def\M{{\rm M}}

\def\Rep{{\mathfrak{Rep}}}

\begin{document}

\title[Property $(T)$ for noncommutative universal lattices]
{Property $(T)$ for noncommutative universal lattices}
\author{Mikhail Ershov}
\address{University of Virginia}
\email{ershov@virginia.edu}
\author{Andrei Jaikin-Zapirain}
\thanks{The second author is  supported by the Spanish Ministry of Science
and Education, grants MTM2007-65385 and MTM2008-06680. This project
was initiated during the second author's visit to University of Virginia
in March, 2008. The authors thank University of Virginia for its hospitality.}
\address{Departamento de Matematicas UAM; Instituto de Ciencias Matematicas
CSIC-UAM-UC3M-UCM}

\email{andrei.jaikin@uam.es}  \maketitle \vspace {-.5cm}
\begin{abstract} We establish a new spectral criterion for Kazhdan's property $(T)$
which is applicable to a large class of discrete groups defined by generators and relations.
As the main application, we prove property $(T)$ for the groups $EL_n(R)$, where $n\geq 3$
and $R$ is an arbitrary finitely generated associative ring. We also strengthen some of the results
on property $(T)$ for Kac-Moody groups from \cite{DJ}.
\end{abstract}
\section{Introduction}
\subsection{The main result}
In this paper we develop a method which can be used to
establish Kazhdan's property $(T)$ for a large class of discrete groups
defined by generators and relations. The paper grew out of an
attempt to find an ``algebraic form'' of an approach to property $(T)$
in a paper of Dymara and Januszkiewicz~\cite{DJ}, which applied
to a class of groups acting on Kac-Moody buildings. It turned
out that not only the proof in \cite{DJ} can be expressed in a purely
group-theoretic language, but it also admits
generalizations of various kinds, which yield many new examples of Kazhdan groups.
The most general criterion for property $(T)$ established in this
paper deals with groups associated with a graph of groups
over a finite graph $Y$, and is applicable whenever the first
eigenvalue of certain ``weighted Laplacian'' of $Y$ is sufficiently
large (see Section~5). However, it is the special cases of that criterion
and their variations which are of most interest. The main
application of our method and the main result of this paper
is the following theorem:

\begin{Theorem}
\label{thm:main} Let $R$ be a finitely generated (associative)
ring with $1$ and $n\geq 3$. Let $G=EL_n(R)$, that is, the subgroup
of $GL_n(R)$ generated by elementary matrices. Then $G$ has
Kazhdan's property $(T)$.
\end{Theorem}
\vskip -.1cm
\begin{Remark}
Theorem~\ref{thm:main} is equivalent to the statement
about property $(T)$ for $EL_n(R)$ where $R=\dbZ\la x_1,\ldots, x_d\ra$
is a free associative ring in finitely many variables. Following~\cite{Ka2},
we call the groups $EL_n(\dbZ\la x_1,\ldots, x_d\ra)$ {\it noncommutative universal lattices}.
\end{Remark}

In fact, we prove property $(T)$ for an even larger group $St_n(R)$,
the Steinberg group over $R$, which naturally surjects onto $EL_n(R)$.
The proof yields an explicit Kazhdan constant for $St_n(R)$
with respect to a natural finite generating set;
asymptotically this constant is $O(\frac{1}{\sqrt{n+d}})$ where $d$ is the
minimal number of generators of $R$ (see Theorem~\ref{Steinberg}).
For a fixed $d$, this bound is asymptotically optimal (see, e.g., \cite{Ka2}).

Our argument produces an explicit finitely presented group with
$(T)$ which surjects onto $EL_n(R)$ if $n\ge 4$ or $R$ is an algebra
over the finite field $\Fq$ for some $q\geq 5$: in the former case
the group $St_n(\dbZ\la x_1,\ldots, x_d\ra)$ has such property,
and in the latter case, we define such a group in our proof. In
particular, we construct an explicit finitely presented group with $(T)$
surjecting onto $EL_3(\Fq[t])=SL_3(\Fq[t])$ (for $q\geq 5$),
answering a question of Shalom~\cite{Sh3}.

Prior to this paper property $(T)$ for $EL_n(R)$, $n\geq 3$, was
known when either $R$ is commutative or the stable range of $R$ is
at most $n$ -- this has been established in the works of
Shalom~\cite{Sh3} and Vaserstein~\cite{Va}, with explicit Kazhdan
constants provided by Ozawa~\cite{Oz}. While all these three
papers are very recent, the study of Kazhdan property for
$EL_n(R)$ has a long history, which we discuss next.

\subsection{Property $(T)$ for $EL_n(R)$: history of the problem}

Property $(T)$ for the groups $SL_n(\dbZ)$ and $SL_n(\dbF_q[t])$,
with $n\geq 3$, has been known since the seminal work of
Kazhdan~\cite{Kazh}. However, Kazhdan's argument was not direct
and relied on the fact that these groups were lattices in
higher-rank simple Lie groups over local fields for which, in
turn, property $(T)$ was verified. In particular, the proof in
\cite{Kazh} did not yield any explicit Kazhdan constants. The
first ``systematic'' approach to the problem of proving property
$(T)$ for Chevalley groups over general commutative rings and
computing Kazhdan constants was given by Shalom~\cite{Sh1} and was
based on the brilliant idea of using bounded generation.
Generalizing a result of Burger~\cite{Bu}, Shalom proved that for
a finitely generated commutative ring $R$, the pair
$(EL_2(R)\ltimes R^2,R^2)$ has relative property $(T)$. It
followed that whenever $R$ is such a ring and the group $EL_n(R)$,
with $n\geq 3$, is known to have bounded elementary generation
property, it must also have property $(T)$, with explicit Kazhdan
constant. In particular, using the fact that
$SL_n(\dbZ)=EL_n(\dbZ)$ was known to be boundedly generated by
$O(n^2)$ elementary subgroups, Shalom has shown that the Kazhdan
constant of $SL_n(\dbZ)$ with respect to a natural generating set
is $O(\frac{1}{n^2})$. In \cite{Ka1}, Kassabov used the techniques
from \cite{Sh1} and a clever combinatorial argument to improve the
Kazhdan constant for $SL_n(\dbZ)$ to $O(\frac{1}{\sqrt{n}})$, a
bound which is asymptotically optimal.

Until very recently, virtually nothing was known about the Kazhdan property for Chevalley groups over rings
of Krull dimension at least two. The first major result in this direction is due to
Kassabov and Nikolov~\cite{KN}, who showed that the groups $SL_n(\dbZ[x_1,\ldots, x_m])$, with $n\geq 3$,
have property $(\tau)$, a weaker version of property $(T)$.
The proof was based on new $K$-theoretic results and
uniform bounded elementary generation for finite congruence quotients
of $SL_n(\dbZ[x_1,\ldots, x_m])$. Shortly afterwards, Kassabov~\cite{Ka2} established
relative property $(T)$ for the pair $(EL_2(R)\ltimes R^2,R^2)$, where $R$ is an arbitrary finitely generated
ring with $1$ (not necessarily commutative), generalizing Shalom's theorem from \cite{Sh1}.
Thus, the reduction of property $(T)$ to bounded elementary generation for $EL_n(R)$
was extended to the case non-commutative rings.

To the best of our knowledge, no new results on bounded elementary
generation of $EL_n(R)$ have been obtained since then. However, in
\cite{Sh3}, Shalom generalized his method from \cite{Sh1} and
showed that $EL_n(R)$ has property $(T)$ as long as every matrix
in $EL_n(R)$ can be reduced to a matrix from $EL_{n}(R)\cap
GL_{n-1}(R)$ using a uniformly bounded number of elementary
transformations. This is a weaker property than elementary bounded
generation and is easily seen to hold whenever ${\rm sr}(R)\leq
n$, where ${\rm sr}(R)$ is the stable range of $R$. If $R$ is
commutative, then ${\rm sr}(R)\leq {\rm Kdim}(R)+2$ by a theorem
of Bass, and thus Shalom's argument yields property $(T)$ for
$EL_n(R)$, where $R$ is any commutative ring of Krull dimension at
most $n-2$ (and $n\geq 3$). Vaserstein~\cite{Va} proved that the
desired ``elementary bounded reduction'' property holds for
$EL_n(R)$, $n\geq 3$, over any commutative Noetherian ring $R$ of
finite Krull dimension. The latter implies property $(T)$ for
$EL_n(R)$, where $n\geq 3$ and $R$ is an arbitrary finitely
generated commutative ring. Finally, we note that the argument in
\cite{Sh3} does not provide explicit Kazhdan constants; however,
this problem has been resolved by Ozawa \cite{Oz} who found a
``quantitative'' version of Shalom's proof.

\subsection{Algebraization and generalization of the method of Dymara and Januszkiewicz}
The bounded generation method of Shalom and its generalizations
are usually considered to be algebraic methods. They are often
contrasted with a broadly defined geometric approach to property
$(T)$ which is applicable to groups acting on buildings with
certain spectral conditions on 1-dimensional links.
In \cite{DJ}, Dymara and Januszkiewicz established property $(T)$
for a class of groups acting on Kac-Moody buildings, using
the notion of $\eps$-orthogonality.\footnote{The same notion
under a different name was previously introduced by Burger~\cite{Bu}}
While this method is presented in geometric terms in \cite{DJ}, many
results from \cite{DJ} concerning property $(T)$
can be formulated (and proved) in a purely group-theoretic setting.
Such an algebraization of Dymara-Januszkiewicz's method is presented in
Section~3 of our paper, with some auxiliary results on Hilbert
space geometry obtained in Section~2. We now briefly describe this
approach.

Let $G$ be a group generated by two subgroups $H$ and $K$. Suppose
we know that $H$ and $K$ have $(T)$, and we want to prove that $G$
has $(T)$. Then we need to show that for any unitary
representation $V$ of $G$ without invariant vectors, a vector from
$V^H$ (the fixed subspace of $H$) cannot be arbitrarily close to a
vector from $V^K$. Closeness between $V^H$ and $V^K$ is measured
by the quantity $\eps(H,K;V)=\sup\{\frac{\la u,v\ra}{\|u\|\|v\|}:
u\in V^H\setminus\{0\},v\in V^K\setminus\{0\}\}$, and we say that
$H$ and $K$ are $\eps$-orthogonal for some $\eps\in\dbR$ if
$\eps(H,K;V)\leq \eps$ for any unitary representation $V$ of $G$
with $V^G=\{0\}$.
%The same notion with other name was previously
%introduced by Burger  in \cite{Bu}. Burger also calculated
%explicitly  $\eps(H,K;V)$ when $H$ and $K$ are the subgroups of
%upper and low triangular matrices respectively in the groups
%$G=SL_2(\Z)$ and $G=SL_3(\Z)$ and $V$ a representation which is
%trivial on some congruence subgroup of $G$.

It is easy to show that $G=\la H,K \ra$ has $(T)$ as long as $H$
and $K$ are $\eps$-orthogonal for some $\eps<1$. This criterion is
very hard to apply directly to an infinite group $G$ since it
requires detailed knowledge of the representation theory of $G$.
What one can effectively apply is the following generalization: if
$G$ is generated by $n$ subgroups $H_1,\ldots, H_n$ such that each
$H_i$ has $(T)$, and any two subgroups $H_i$ and $H_j$ are
$\eps$-orthogonal for sufficiently small $\eps$ (where
``sufficiently small'' depends on $n$), then $G$ has $(T)$. The
arguments in \cite{DJ} (or rather their group-theoretic
counterparts) prove this statement for $\eps<\frac{1}{7^{n-2}}$,
but this result can be improved in several ways. First, instead of
requiring that each $H_i$ has $(T)$ it suffices to require that
each pair $(G,H_i)$ has relative property $(T)$. Second, the upper
bound on $\eps$ can be significantly improved from
$\frac{1}{7^{n-2}}$ to $\frac{1}{n-1}$. Thus, we obtain the
following result (see Corollary~\ref{Tn} for a more detailed
statement, including explicit Kazhdan constant estimates):
\begin{Theorem}
\label{Thm12}
Let $G$ be a group generated by subgroups $H_1,\ldots, H_n$. Suppose
that the pair $(G,H_i)$ has relative property $(T)$ for each $i$, and any two subgroups
$H_i$ and $H_j$ are $\eps$-orthogonal for some $\eps<\frac{1}{n-1}$. Then $G$ has property $(T)$.
\end{Theorem}
Theorem~\ref{Thm12} for $n=3$ will already be established in Section~3, but
for larger $n$ it will be deduced from our spectral criterion (Theorem~\ref{GG1}).
The key concept which appears in the statement of that criterion (and also
enters the bound for the Kazhdan constant) is that of {\it codistance} between a finite
family of subgroups of a given group (see subsection 2.2 for details).

\subsection{About the proof of Theorem~\ref{thm:main}}
We now give a brief outline of the proof of Theorem~\ref{thm:main}.
If one wants to prove that a group $G$ has $(T)$ using
Theorem~\ref{Thm12} or some variation of it, one should consider the class $\calC$
consisting of subgroups $H$ of $G$ such that $(G,H)$ has relative $(T)$
(note that all finite subgroups of $G$ belong to $\calC$), and then
look for families of pairwise almost orthogonal subgroups within $\calC$
which generate $G$. It is easy to see that any two commuting subgroups
will be $0$-orthogonal; however, using only such
pairs, one cannot construct interesting new examples of groups with $(T)$.
A much deeper result will be obtained in Section~4,
where we show that if $N$ is a group of nilpotency class two
generated by abelian subgroups $X$ and $Y$, then $X$ and $Y$ are
$\frac{1}{\sqrt{2}}$-orthogonal. Moreover, the orthogonality
constant can be improved under additional assumptions on $N$.

Now let $G=EL_n(R)$ where $R$ is a finitely generated associative ring with $1$.
Then $G$ is clearly generated by $n$ abelian subgroups (root subgroups),
each pair of which generates a subgroup of nilpotency class $2$.
A more delicate analysis shows that the same can always be done using just
three abelian subgroups $H_1, H_2, H_3$ such that $H_i$ and $H_j$ are
$\frac{1}{\sqrt{m^{[n/3]}}}$-orthogonal where $m=m(R)$ is the minimal index of a right ideal in $R$.
Thus, $H_1, H_2$ and $H_3$ are pairwise $\eps$-orthogonal for some $\eps<1/2$ whenever $m\geq 5$ or $n\geq 9$.
Using Kassabov's variation of Shalom's theorem on relative property $(T)$ for $(SL_2(\dbZ)\ltimes \dbZ^2, \dbZ^2)$,
we will show that each pair $(G,H_i)$ has relative property $(T)$, and thus deduce
from Theorem~\ref{Thm12} that $G=EL_n(R)$ has $(T)$.
The argument we just sketched proves property $(T)$ not only for the group $EL_n(R)$
(under the additional assumption $m\geq 5$ or $n\geq 9$)
but for an explicit finitely presented cover of it. It was proved by Shalom~\cite{Sh2}
that every discrete group with $(T)$ has a finitely presented cover with $(T)$; however,
to the best of our knowledge, no explicit finitely presented cover with $(T)$ for groups
of the form $EL_n(R)$ was known except for the cases
when $EL_n(R)$ itself was known to be finitely presented and known to have $(T)$
(see the end of subsection 6.1 for a more detailed discussion).

To deal with the remaining cases when both $m$ and $n$ are small, we give a different proof
of property $(T)$ which works for any ring $R$ and any $n\geq 3$ and in fact
yields a better Kazhdan constant. We use a variation of our spectral criterion to prove
that the Kazhdan constant $\kappa(EL_n(R),X)$ is positive where $X$ is the union of
$6$ abelian subgroups. Quite amazingly, the only fact about $EL_n(R)$ used here
is that it is ``graded by a root system of type $A_2$'' in the suitable sense
(see subsection~5.4). Then, as with the other proof, we finish the argument by using
relative property $(T)$ for $(EL_2(R)\ltimes R^2, R^2)$ or rather certain
generalization of it established in \cite{Ka2}.

As an immediate application of Theorem~\ref{thm:main}, we obtain a
simple example of a profinite group $G$ containing dense finitely
generated abstract subgroups $A$ and $B$ such that $A$ is
amenable and $B$ has property $(T)$. This gives a counterexample
to a conjecture of Lubotzky-Weiss \cite[Conjecture~1.2]{LW},
strengthening an earlier result of Kassabov~\cite{Ka2} -- see
subsection 6.3.\footnote{Kassabov (private communication)
constructed another counterexample to Lubotzky-Weiss conjecture
using \cite{Sh3}.}

\subsection{Beyond linear groups} We believe that the method introduced
in this paper has vast potential to produce new examples of
non-linear Kazhdan groups. This problem is addressed in the last
section of this paper, where we introduce a large class of groups
which we call Kac-Moody-Steinberg groups, and show that many of
these groups have property $(T)$. These Kac-Moody-Steinberg groups
are given by simple presentations, and quotients of the groups
from this class include the linear groups $EL_n(R)$ discussed
above as well as parabolic subgroups of Kac-Moody groups with
simply-laced Dynkin diagrams. We use Kac-Moody-Steinberg groups to
construct new examples of Golod-Shafarevich groups with property
$(T)$, improving and generalizing the main result from \cite{Er}.
Finally, we believe that one may be able to establish the
expanding property for some families of finite groups by realizing
them as quotients of Kac-Moody-Steinberg groups.

\subsection{Organization of the paper} In Section~2 we introduce and study the notions
of $\eps$-orthogonality and codistance for subspaces of Hilbert spaces. In section~3
we define basic concepts related to property $(T)$ and describe an approach
to property $(T)$ via $\eps$-orthogonality. In Section~4 we establish several results
about representations of nilpotent groups of class two. In Section~5 we state
and prove our main spectral criterion for property $(T)$ and its variations.
Proof of property $(T)$ for $EL_n(R)$ and some related results are contained in Section~6.
Finally, in Section~7 we discuss Kac-Moody-Steinberg groups.
\vskip .1cm
{\bf Acknowledgements.} We are very grateful to Dmitry Yakubovich for useful conversations
and to Martin Kassabov who carefully read the manuscript and made many useful
suggestions. We also thank Yves de Cornulier and Pierre de la Harpe for
providing references to several results.

\section{Geometry of subspaces in Hilbert spaces}

Throughout the paper all Hilbert spaces are assumed to be complex, and `subspace'
will mean a closed subspace. If $W$ is a subspace of a Hilbert space $V$,
then $W^{\perp V}$ will denote the orthogonal complement of $W$ in $V$;
we will write $W^{\perp}$ for $W^{\perp V}$ when $V$ is clear from the context.
We denote by $P_W:V\to V$ the operator of orthogonal projection onto $W$.
Thus, for any $x\in V$ and any subspace $W$ of $V$, we have $x=P_W(x)+P_{W^{\perp}}(x)$.

\subsection{$\eps$-closeness and $\eps$-orthogonality}

\begin{Definition}\rm Let $V$ be a Hilbert space, and let $X$ and $Y$ be subspaces of $V$.
Let $\epsilon\geq 0$ be a real number.
\begin{itemize}
\item[(a)] We will say that $X$ and $Y$ are $\epsilon$-orthogonal and write $X\perp_{\eps} Y$
if $$|\la x,y\ra|\leq \epsilon \|x\| \|y\|$$
for any $x\in X$ and $y\in Y$.
\item[(b)] We will say that $X$ is $\epsilon$-close to $Y$ if $dist(x,Y)\leq \epsilon\|x\|$ for any $x\in X$,
that is, for any $x\in X$ there exists $y\in Y$ such that $\| x-y\|\leq \epsilon\|x\|$.
\end{itemize}
\end{Definition}
Since in a Hilbert space, the minimal distance from an element $x$ to a subspace $Y$
is via the orthogonal projection to $Y^{\perp}$, the notion of $\epsilon$-closeness can be characterized as follows:

\begin{Proposition} Let $V$ be a Hilbert space, $X$ and $Y$ subspaces of $V$.
Then $X$ is $\epsilon$-close to $Y$ if and only if $\|P_{Y^{\perp}}(x)\|\leq \epsilon\|x\|$ for any $x\in X$.
\end{Proposition}

\begin{Proposition}
\label{prop1}
Let $X$ and $Y$ be subspaces of a Hilbert space $V$. The following are equivalent:

(i) $X$ is $\epsilon$-close to $Y$

(ii) $X$ and $Y^{\perp}$ are $\epsilon$-orthogonal
\end{Proposition}
\begin{proof}
``(i)$\Rightarrow$ (ii)'' Take any $x\in X$ and $y\in Y^{\perp}$, and write $x=u+v$ where $u\in Y$ and $v\in Y^{\perp}$.
Then $\|v\|\leq \epsilon \|x\|$ by assumption, so
$$|\la x,y\ra|=|\la v,y\ra|\leq \epsilon \|x\|\|y\|.$$
\vskip .12cm
``(ii)$\Rightarrow$ (i)'' Let $P=P_{Y^{\perp}}:V\to Y^{\perp}$. Then for any $x\in X$
we have $$|\la x, P(x)\ra|\leq \epsilon\|x\| \|P(x)\|$$ since $X$ and $Y^{\perp}$ are $\epsilon$-orthogonal.
On the other hand, $\la P(x), x-P(x)\ra=0$, so $\|P(x)\|^2=|\la x,P(x)\ra|$, and thus $\|P(x)\|\leq \epsilon \|x\|$.
\end{proof}
\begin{Proposition}
\label{prop2}
Let $X$ and $Y$ be subspaces of a Hilbert space $V$, and suppose that $X$ is $\epsilon$-close to $Y$.
The following hold:

(a) $Y^{\perp}$ is $\epsilon$-close to $X^{\perp}$.

(b) Assume in addition that $\overline{X+Y^{\perp}}=V$. Then $Y$ is $\eps$-close to $X$.
\end{Proposition}
\begin{proof}
(a) follows directly from Proposition~\ref{prop1} and the fact that the relation of $\epsilon$-orthogonality is symmetric.
\vskip .1cm
(b) Let $P:V\to Y$ be the orthogonal projection onto $Y$. Since $\overline{X+Y^{\perp}}=V$, we have $P(X)=Y$. Thus,
for any $y\in Y$ there exists $x\in X$ such that $P(x)=y$. Since $X$ is $\epsilon$-close to $Y$, we have
\begin{equation}
\|x-y\|=\|x-P(x)\|\leq \epsilon \|x\|.
\label{eq1}
\end{equation}
Now project $y$ onto the one-dimensional space $\dbC x$, that is, write $y=\lam x + v$ with $v\perp x$.
We will show that $\| v\|\leq \eps \|y\|$, which would imply that $y$ is $\eps$-close to $X$ and finish the proof.
Indeed, since $0=\la v,x\ra=\la y-\lam x,x\ra$, we have $\lam=\frac{\la y,x \ra}{\|x\|^2}$. Furthermore,
$(x-y)\perp y$, so $\la y,x\ra=\|y\|^2$ and $\lam=\frac{\|y\|^2}{\|x\|^2}$. Thus,
\begin{multline*}
\| v\|^2=\| y\|^2-\lam^2 \|x\|^2=\| y\|^2-\frac{\|y\|^4}{\|x\|^2}=\frac{\|y\|^2}{\|x\|^2}(\| x\|^2-\|y\|^2)=\\
\frac{\|y\|^2\cdot\|x-y\|^2}{\|x\|^2}\leq \eps^2\| y\|^2 \mbox{ by } \eqref{eq1}.
\end{multline*}
\end{proof}
The following two lemmas will play a key role in the next section. Their proofs
are similar to those of \cite[Sublemma 4.8]{DJ} and \cite[Sublemma 4.10]{DJ}, respectively.

\begin{Lemma}
\label{lem1}
Let $X$ and $Y$ be subspaces of a Hilbert space $U$. Suppose that $X$ and $Y$ are $\epsilon$-orthogonal
and $\overline{X+Y}=U$. Then $X^{\perp}$ and $Y^{\perp}$ are $\eps$-orthogonal.
\end{Lemma}
\begin{proof} This is simply a combination of Proposition~\ref{prop1} and Proposition~\ref{prop2} applied
to the pair $\{X,Y^{\perp}\}$.
\end{proof}

\begin{Lemma}
\label{lem2}
Let $X,Y$ and $Z$ be subspaces of a Hilbert space $V$. Suppose that $X\perp_{\eps_3}Y,\,\,$
$X\perp_{\eps_2}Z$ and, $Y\perp_{\eps_1}Z$ for some $\eps_1, \eps_2, \eps_3<1$.
Then the subspaces $X+Y$ and $Z$ are $\eps_0$-orthogonal, where
$\eps_0=\frac{\sqrt{2}\cdot\max\{\eps_1,\eps_2\}}{\sqrt{1-\eps_3}}$.
\end{Lemma}
\begin{Remark} Note that $X+Y$ is closed since $X\perp_{\eps_3}Y$ with $\eps_3<1$.
\end{Remark}
\begin{proof} Take any $x\in X$, $y\in Y$, and write $x=x_Z+x_{Z^{\perp}}$, $y=y_Z+y_{Z^{\perp}}$,
where $x_Z, y_Z\in Z$ and $x_{Z^{\perp}}, y_{Z^{\perp}}\in Z^{\perp}$. Let $\eps=\max\{\eps_1,\eps_2\}$.
Since $Z$ is $\epsilon$-orthogonal to both $X$ and $Y$, we have $\|x_Z\|\leq \eps \|x\|$ and $\|y_Z\|\leq \eps \|y\|$.
Therefore, $$\|x_Z+y_Z\|^2\leq 2(\|x_Z\|^2+\|y_Z\|^2)\leq 2\,\eps^2(\|x\|^2+\|y\|^2).$$

On the other hand, $|\la x,y\ra|\leq \eps_3\|x\|\|y\|$, and therefore,
$$\|x+y\|^2\geq \|x\|^2+\|y\|^2-2|\la x,y\ra|\geq (1-\epsilon_3)(\|x\|^2+\|y\|^2).$$ Combining the two inequalities,
we get that $$\|(x+y)_Z\|=\|x_Z+y_Z\|\leq \eps_0 \|x+y\|$$ where $\eps_0$ is as in the statement of the Lemma.
Thus, $X+Y$ is $\eps_0$-orthogonal to $Z$.
\end{proof}

\subsection{Codistance between subspaces}

\begin{Definition}\rm Let $U$ and $W$ be  subspaces of a Hilbert
space $V$. We put $$\eps(U,W)=\sup \{\|P_{W}(u)\| : u\in U,
\|u\|=1\}=\sup \left\{\frac{|\la u,w\ra|}{\|u\|\|w\|} \,:\, 0\ne u\in U,
0\ne w\in W\right\}.$$ and call it the {\it orthogonality constant }
between $U$ and $W$. Thus, $\eps(U,W)$ is the smallest $\eps$ for which
$U$ and $W$ are $\eps$-orthogonal.
\end{Definition}

\begin{Lemma}\label{ort} Let $U$ and $W$ be  subspaces of a Hilbert
space $V$.
 Suppose that $V=\overline{U+W}$ and $U\cap W=\{0\}$.
Then $\eps(U,W)= \eps(U^\perp, W^\perp)$.
\end{Lemma}
\begin{proof}
Since $V=\overline{U+W}$, Lemma~\ref{lem1} implies that $\eps(U^\perp, W^\perp)\leq \eps(U,W)$.
On the other hand, $U\cap W=\{0\}$ implies that $\overline{U^\perp+W^\perp}=V$, which
yields the reverse inequality.
\end{proof}

Now we will introduce the notion of codistance between a finite collection of subspaces.

\begin{Definition}\rm Let $V$ be a Hilbert space, and let $\{U_i\}_{i=1}^n$ be subspaces of $V$.
Consider the Hilbert space $V^n$ and
its subspaces $U_1\times U_2\times\ldots\times U_n$ and $diag(V)=\{(v,v,\ldots,v) : v\in V\}$.
The quantity $$\rho(\{U_i\})=\left(\eps(U_1\times\ldots\times U_n,diag(V))\right)^2$$ will be called the
{\it codistance } between the subspaces $\{U_i\}_{i=1}^n$. It is easy to see that
$$\rho(\{U_i\})=\sup\left\{\frac{\|u_1+\cdots+u_n\|^2}{n(\|u_1\|^2+\cdots+\|u_n\|^2)} : u_i\in U_i\right\}.$$
\end{Definition}
For any collection of $n$ subspaces $\{U_i\}_{i=1}^n$ we have
$\frac{1}{n}\leq \rho(\{U_i\})\leq 1$, and $\rho(\{U_i\})=\frac{1}{n}$ if and only if
$\{U_i\}$ are pairwise orthogonal. In the case of two subspaces we have an
obvious relation betwen $\rho(U,W)$ and $\eps(U,W)$:
$$\rho(U,W)=\frac{1+\eps(U,W)}{2}.$$

We finish this section with a technical lemma which will be needed for explicit
estimation of Kazhdan constants:
\begin{Lemma}
\label{Kazhdanprep}
Let $V$ be a Hilbert space, $\{V_i\}_{i=1}^n$ subspaces of $V$, and let $\rho=\rho(\{V_i\})$.
Let $x\in V$, and for each $i\in \{1,\ldots,n \}$, write $x=a_i+b_i$, with
$a_i\in V_i$ and $b_i\in V_i^{\perp}$.
Then there exists $j$ such that $\|b_j\|\geq \sqrt{1-\rho}\,\|x\|$.
\end{Lemma}
\begin{proof} Let $a=(a_1,a_2,\ldots, a_n)\in V^n$ and $y=(x,x,\ldots,x)\in V^n$.
Then $a\in V_1\times V_2\times\ldots\times V_n$ and $y\in \diag(V)$.
By definition of $\rho=\rho(\{V_i\})$, we have
\begin{equation}
\label{BBB}
\|\la a,y\ra\|^2\leq \rho \|a\|^2 \|y\|^2.
\end{equation}
Since $\la a,y\ra=\sum \la a_i,x_i\ra=\sum \|a_i\|^2=\|a\|^2$
and $\|y\|^2=n\|x\|^2$, \eqref{BBB} yields
$$\sum_{i=1}^n\|a_i\|^2\le \rho n\|x\|^2.$$
Therefore, $$\sum_{i=1}^n \|b_i\|^2=
\sum_{i=1}^n (\|x\|^2-\|a_i\|^2)=
n\|x\|^2-\sum_{i=1}^n \|a_i\|^2\ge n(1-\rho)\|x\|^2,$$
and thus $\|b_j\|^2\geq (1-\rho) \|x\|^2$ for some $j$.
\end{proof}

\section{$\epsilon$-orthogonality and property $(T)$}
In this section we show how the notions of $\epsilon$-orthogonality and codistance
can be used to establish property $(T)$.
We are primarily interested
in the case of discrete groups, but many definitions and results
will be stated for arbitrary topological groups.

For a (topological) group $G$, by $\Rep(G)$ we will denote the class of (continuous)
unitary representations of $G$, and by $\Rep_0(G)$ the class of (continuous)
unitary representations of $G$ without nonzero invariant vectors.

\subsection{Basic definitions}
\begin{Definition}\rm Let $G$ be a group and $S$ a subset of $G$.
\begin{itemize}
\item[(a)] Let $V\in\Rep(G)$. A nonzero vector $v\in V$ will be
called $(S,\eps)$-invariant if
$$\|sv-v\|\leq \epsilon\|v\| \mbox{ for any } s\in S.$$
\item[(b)] Let $V\in \Rep_0(G)$. The {\it Kazhdan constant
$\kappa(G,S,V)$ } is the infimum of the set
$$\{\epsilon> 0 : V \mbox{ contains an } (S,\epsilon)\mbox{-invariant vector.}\}$$
\item[(c)] The {\it Kazhdan constant $\kappa(G,S)$ }of $G$ with
respect to $S$ is the infimum of the set $\{\kappa(G,S,V)\}$ where
$V$ runs over $\Rep_0(G)$.
\end{itemize}
\end{Definition}
\begin{Definition}\rm
A discrete group $G$ is said to have {\it property $(T)$ } if
$\kappa(G,S)>0$ for some finite subset $S$ of $G$.
\end{Definition}
If $G$ is discrete, it is known that $\kappa(G,S)$ may only be
nonzero if $S$ is a generating set for $G$ (see, e.g. \cite[Proposition 1.3.2]{BHV}).
Thus, a discrete group $G$ with property $(T)$ is automatically finitely generated.
Furthermore, if $G$ has property $(T)$, then $\kappa(G,S)>0$ for
any finite generating set $S$ of $G$, but the Kazhdan constant
$\kappa(G,S)$ depends on $S$. \vskip .1cm

Property $(T)$ for a group $G$ can often be proved by first showing that $\kappa(G,B)>0$ for
some infinite subset $B$ of $G$ and then establishing relative property
$(T)$ for the pair $(G,B)$. We now explain how this can be done.

Relative property $(T)$ has been originally defined for pairs $(G,H)$
where $H$ is a normal subgroup of $G$:

\begin{Definition}\rm Let $G$ be a discrete group and $H$ a normal subgroup of $G$.
The pair $(G,H)$ has {\it relative property $(T)$} if there exist a finite set $S$
and $\eps>0$ such that if $V$ is any unitary representation of $G$
with $(S,\eps)$-invariant vector, then $V$ has a (nonzero) $H$-invariant vector.
The largest $\eps$ with this property (for a fixed set $S$) is called
the {\it relative Kazhdan constant} of $(G,H)$ with respect to $S$ and denoted by
$\kappa(G,H;S)$.
\end{Definition}

More recently, the notion of relative property $(T)$ has been generalized
to pairs $(G,B)$ where $B$ is an arbitrary subset of a group $G$ (see \cite{Co}).
For our purposes, it is most convenient to define relative property $(T)$
as follows:
\footnote{\cite[Theorem~1.1]{Co} gives a list of six equivalent conditions,
each of which can be taken as the definition of relative property $(T)$.
Our definition appears to be a stronger version of condition (3) on that
list, but it is actually equivalent to (3), as the proof of \cite[Theorem~1.1]{Co}
shows. We thank Yves de Cornulier for pointing this out to us.}

\begin{Definition}\rm Let $G$ be a discrete group and $B$ a subset of $G$.
The pair $(G,B)$ has {\it relative property $(T)$} if for any
$\epsilon>0$ there are a finite subset $S$ of $G$  and $\mu>0$
such that if $V$ is any unitary representation of $G$ and $v\in V$ is
$(S,\mu)$-invariant, then $v$ is $(B,\eps)$-invariant.
\end{Definition}
\begin{Remark} The pair $(G,B)$ may have relative property $(T)$
even if $G$ is not finitely generated: for instance, if $B$ is a
group with property $(T)$, then $(G,B)$ has relative property
$(T)$ for any overgroup $G$. However, if $G$ is generated by a
finite set $S_0$, then in the definition of relative $(T)$ for
$(G,B)$ we can require that $S$ equals $S_0$.
\end{Remark}
\vskip .1cm

An important special case of relative property $(T)$ is when the dependence of
$\mu$ on $\eps$ in the above definition may be expressed by a linear function.
We reflect this property in the following definition.
\begin{Definition}\rm Let $G$ be a discrete group and $B$ and $S$ subsets of $G$.
 The {\it Kazhdan ratio $\kappa_r(G,B;S)$ } is the
largest $\delta\in\dbR$ with the following property: if $V$ is any
unitary representation of $G$ and $v\in V$ is
$(S,\delta\eps)$-invariant, then $v$ is
$(B,\eps)$-invariant.
\end{Definition}
Clearly, if $\kappa_r(G,B;S)>0$ for some finite set $S$, then $(G,B)$ has relative $(T)$.
On the other hand, if $B$ is a normal subgroup of $G$, then $\kappa_r(G,B;S)\geq \frac{\kappa(G,B,S)}{2}$
(this inequality is essentially established in \cite[Corollary~2.3]{Sh1}). Thus,
if $B$ is a normal subgroup of $G$, then relative property $(T)$ for $(G,B)$
is equivalent to the positivity of the Kazhdan ratio $\kappa_r(G,B;S)$ for some finite set $S$.

It is clear from definitions that if a group $G$ has a subset $B$ such that
$\kappa(G,B)>0$ and $(G,B)$ has relative property $(T)$, then $G$ has property $(T)$.
If in addition, we know that $\kappa_r(G,B;S)>0$ for some finite set $S$,
we can estimate the Kazhdan constant $\kappa(G,S)$ using the following
obvious inequality:
\begin{equation}
\label{CCC} \kappa(G,S)\geq \kappa(G,B)\kappa_r(G,B;S)
\end{equation}
This argument was used in Shalom's
proof of property $(T)$ for $SL_n(\dbZ)$ as follows. In
\cite{Sh1}, it is first shown that $\kappa(SL_2(\dbZ)\ltimes \dbZ^2,\dbZ^2; F)\geq \frac{1}{10}$
for some natural generating set $F$ of $SL_2(\dbZ)\ltimes \dbZ^2$, and thus
$\kappa_r(SL_2(\dbZ)\ltimes \dbZ^2,\dbZ^2; F)\geq \frac{1}{20}$
(since $\dbZ^2$ is a normal subgroup of $SL_2(\dbZ)\ltimes \dbZ^2$).
Using natural embeddings of
$SL_2(\dbZ)\ltimes \dbZ^2$ into $SL_n(\dbZ)$, one concludes that
$\kappa_r (SL_n(\dbZ), U;\Sigma)\geq \frac{1}{20}$, where $\Sigma$
is the set of elementary matrices with $1$ off the diagonal and $U$ is
the set of all elementary matrices in $SL_n(\dbZ)$. On the other
hand, since $SL_n(\dbZ)$ is boundedly generated by elementary
matrices, we have $\kappa(SL_n(\dbZ),U)>0$. Thus, by
\eqref{CCC}, $SL_n(\dbZ)$ has property $(T)$.

The proof of Theorem~\ref{thm:main} will use similar logic.
However, in our case showing that $\kappa(EL_n(R),U)>0$, where $U$ is the set of all elementary
matrices in $EL_n(R)$, is not easy; in fact, this will be the main part of the proof.
Once this is done, we use relative property $(T)$ for the pair $(EL_2(R)\ltimes R^2,R^2)$
established by Kassabov, and then adapt the argument in the previous paragraph.

\subsection{Proving property $(T)$ using $\eps$-orthogonality}
We start by defining the notions of $\eps$-orthogonality and codistance for subgroups
of a given group. If $V$ is a representation of a group $G$ and $H$ is a subgroup of $G$,
by $V^H$ we will denote the set of $H$-invariant vectors.

\begin{Definition}\rm Let $G$ be a group.
\begin{itemize}
\item[(a)] Let $H$ and $K$ be subgroups of $G$ such that $G=\la
H,K\ra$. We will say that $H$ and $K$ are $\eps$-orthogonal if for
any $V\in\Rep_0(G)$,
the subspaces $V^H$ and $V^K$ are $\eps$-orthogonal.
\item[(b)]
Let $\{H_i\}_{i=1}^n$ be subgroups of $G$. The {\it codistance }
between $\{H_i\}$ in $G$, denoted $\rho(\{H_i\},G)$, is defined to
be the supremum of the set
$$\{\rho(V^{H_1},\ldots, V^{H_n}): V\in\Rep_0(G)\}.$$
If  $G=\la H_1,\ldots, H_n\ra$ we simply write
$\rho(\{H_i\})$ instead of $\rho(\{H_i\},G)$.
\end{itemize}
\end{Definition}
\begin{Remark} It is easy to see that if $G\neq \la H_1,\ldots, H_n\ra$,
then $\rho(\{H_i\},G)=1$.
\end{Remark}

\begin{Lemma}
\label{orthogT}
Let $G$ be a group and $H_1,H_2,\ldots, H_n$ subgroups of $G$ such that $G=\la H_1,\ldots, H_n\ra$.
Let $\rho=\rho(\{H_i\})$, and suppose that $\rho<1$. The following hold:
\begin{itemize}
\item[(a)] $\kappa(G,\bigcup H_i)\geq\sqrt{2(1-\rho)}. $
\item[(b)] Let $S_i$ be a generating set of $H_i$, and
let $\delta=\min\{\kappa(H_i,S_i)\}_{i=1}^n$. Then
$$\kappa(G,\bigcup S_i)\geq\delta\,\sqrt{1-\rho}. $$
\item[(c)] Assume in addition that each pair $(G,H_i)$ has relative property $(T)$.
Then $G$ has property $(T)$.
\end{itemize}
\end{Lemma}
\begin{proof} It is clear from definitions that if each pair $(G,H_i)$ has relative $(T)$,
then $(G,\cup H_i)$ also has relative $(T)$. Hence (c) is a consequence of (a).
By \cite[Proposition~1.1.5]{BHV}, $\kappa(\Gamma,\Gamma)\geq \sqrt{2}$ for any group $\Gamma$,
so (a) is a special case of (b) with $S_i=H_i$. Thus, we only need to prove (b).

Let $V\in\Rep_0(G)$, and take any nonzero $v\in V$.
For each $i=1,\ldots, n$ we write $v=a_i+b_i$ where $a_i\in V^{H_i}$ and $b_i\in (V^{H_i})^{\perp}$.
By Lemma~\ref{Kazhdanprep}, we have $\|b_i\|\geq \|v\|\sqrt{1-\rho}$ for some $i$.

Note that $(V^{H_i})^{\perp}$ is a unitary representation of $H$
without invariant vectors. Since $\kappa(H_i,S_i)\geq \delta$ and
$b_i\in (V^{H_i})^{\perp}$, there exists $s\in S_i$ such that
$\|sb_i-b_i\|\geq \delta\|b_i\|\geq \delta\|v\|\sqrt{1-\rho}$. On
the other hand,
$$\|sv-v\|=\|s(a_i+b_i)-(a_i+b_i)\|=\|sb_i-b_i\|$$ since $a_i$ is $s$-invariant.
Thus, $\kappa(G,\bigcup S_i, V)\geq \delta\sqrt{1-\rho}$.
\end{proof}

%The same notion with other name was previously
%introduced by Burger  in \cite{Bu}. Burger also calculated
%explicitly  $\eps(H,K;V)$ when $H$ and $K$ are the subgroups of
%upper and low triangular matrices respectively in the groups
%$G=SL_2(\Z)$ and $G=SL_3(\Z)$ and $V$ a representation which is
%trivial on some congruence subgroup of $G$.

In \cite{Bu}, Burger used a variation of Lemma~\ref{orthogT}
and explicit lower bounds for some orthogonality constants
to estimate Kazhdan constants of $G=SL_3(\dbZ)$ with respect
to some class of unitary representations of $G$, including all
finite-dimensional irreducible ones. His argument, however,
did not apply to all representations without nonzero invariant
vectors and thus did not yield a new proof of property $(T)$ for $SL_3(\dbZ)$.

In general, it seems very hard to establish property $(T)$
for an infinite group $G$ by a direct application of Lemma~\ref{orthogT}.
However, the next two results make it possible to prove property $(T)$
for some complicated infinite groups, only applying Lemma~\ref{orthogT} to
much simpler groups whose representations are easily described.
This idea was introduced by Dymara and Januszkiewicz in \cite{DJ}.
Both Proposition~\ref{3subgps} and Corollary~\ref{T3} below have direct
counterparts in \cite{DJ}, but we work in a different context and obtain sharper estimates.

\begin{Proposition}
\label{3subgps} Let $G$ be a group, and let $H_1, H_2, H_3$ be subgroups of $G$ such that
$G=\la H_1, H_2, H_3\ra$. Assume that $H_1$ and $H_2$ are $\eps_3$-orthogonal,
$H_1$ and $H_3$ are $\eps_2$-orthogonal and $H_2$ and $H_3$ are $\eps_1$-orthogonal
for some $\eps_1,\eps_2,\eps_3<1$. Then the subgroups $\la H_1, H_3\ra$ and $\la H_1, H_2\ra$
are $\epsilon_0$-orthogonal where $$\eps_0=\frac{\sqrt{2}\cdot\max\{\eps_1,\eps_2\}}{\sqrt{1-\eps_3}}.$$
\end{Proposition}
\begin{proof}
Let $V\in\Rep_0(G)$.
We put \vskip .1cm \centerline{$V_1=V^{\la H_2, H_3\ra}$,
$V_2=V^{\la H_1, H_3\ra}$ and $V_3=V^{\la H_1, H_2\ra}$.} \vskip
.1cm Let $V_0=\overline{V_1+V_2+V_3}$. All subsequent computations
will be done inside $V_0$, so for any subset $W$ of $V_0$ we set
$W^{\perp}=W^{\perp V_0}$. \vskip .1cm

Note that $V_3^{\perp V}\in\Rep_0(\la H_1, H_2\ra)$.
Furthermore, we have $(\overline{V_3+V_2})\cap V_3^{\perp}\subseteq (V_3^{\perp})^{H_1}$
and $(\overline{V_3+V_1})\cap V_3^{\perp}\subseteq (V_3^{\perp})^{H_2}$. By $\eps_3$-orthogonality
of $H_1$ and $H_2$, we have
\begin{equation}
\label{firstrel}
(\overline{V_3+V_2})\cap V_3^{\perp}\perp_{\eps_3}(\overline{V_3+V_1})\cap V_3^{\perp}.
\end{equation}
Note that $(\overline{V_3+V_2})\cap V_3^{\perp}+(\overline{V_3+V_1})\cap V_3^{\perp}=V_3^{\perp}$.
Indeed, if $P:U\to V_3^{\perp}$ is the orthogonal projection, then
$(\overline{V_3+V_2})\cap V_3^{\perp}=P(V_2)$, $(\overline{V_3+V_1})\cap V_3^{\perp}=P(V_1)$,
and $P(V_3)=0$; on the other hand $P(V_0)=\overline{P(V_1)+P(V_2)+P(V_3)}$.
\vskip .2cm

We can apply Lemma~\ref{lem1} with $X=(\overline{V_3+V_2})\cap V_3^{\perp}$, $Y=(\overline{V_3+V_1})\cap V_3^{\perp}$
and $U=X+Y=V_3^{\perp}$. It is clear that $X^{\perp V_3^{\perp}}=V_3^{\perp}\cap V_2^{\perp}$ and
$Y^{\perp V_3^{\perp}}=V_3^{\perp}\cap V_2^{\perp}$, and thus we get
\begin{equation}
\label{secondrel}
(V_3^{\perp}\cap V_2^{\perp})\perp_{\eps_3}(V_3^{\perp}\cap V_1^{\perp}).
\end{equation}
By the same argument, we obtain that
$$(V_2^{\perp}\cap V_1^{\perp})\perp_{\eps_2}(V_2^{\perp}\cap V_3^{\perp}) \mbox{ and }
(V_1^{\perp}\cap V_2^{\perp})\perp_{\eps_1}(V_1^{\perp}\cap V_3^{\perp}).
$$
Now we apply Lemma~\ref{lem2} to the subspaces $X'=V_2^{\perp}\cap V_3^{\perp}$,
$Y'=V_1^{\perp}\cap V_3^{\perp}$ and $Z'=V_1^{\perp}\cap V_2^{\perp}$.
Note that $X'+Y'=(\overline{V_3+V_1}\cap\overline{V_3+V_2})^{\perp}=V_3^{\perp}$
(otherwise, we would get $\overline{V_3+V_1}\cap\overline{V_3+V_2}\cap V_3^{\perp}\neq \{0\}$,
contrary to \eqref{firstrel}). Thus we get
\begin{equation}
\label{thirdrel}
V_3^{\perp}\perp_{\eps_0}(V_1^{\perp}\cap V_2^{\perp}).
\end{equation}
where $\eps_0=\frac{\sqrt{2}\cdot\max\{\eps_1,\eps_2\}}{\sqrt{1-\eps_3}}$.
Finally, since $V_3^{\perp}+V_1^{\perp}\cap V_2^{\perp}=V_0$,
applying Lemma~\ref{lem1} again, we get
\begin{equation}
\label{fourthrel}
V_3\perp_{\eps_0}\overline{V_1+V_2}.
\end{equation}
In particular, $V_3\perp_{\eps_0} V_2$.
\end{proof}

\begin{Corollary}
\label{T3}
Let $G$ be a group, and let $H_1, H_2, H_3$ be subgroups of $G$ such that
$G=\la H_1, H_2, H_3\ra$. Assume that there exist $\eps_1,\eps_2,\eps_3$ such that
$H_1$ and $H_2$ are $\eps_3$-orthogonal,
$H_2$ and $H_3$ are $\eps_1$-orthogonal,
$H_3$ and $H_1$ are $\eps_2$-orthogonal,
and $$\frac{\sqrt{2}\max\{\eps_1,\eps_2\}}{\sqrt{1-\eps_3}}<1$$
(note that this inequality holds whenever each $\eps_i<\frac{1}{2}$).
Let
$\eps'=\max\{\eps_1,\eps_2,\eps_3\}$ and $\eps_0=\frac{\sqrt{2}\max\{\eps_1,\eps_2\}}{\sqrt{1-\eps_3}}$.
The following hold:
\begin{itemize}
\item[(a)] $\kappa(G,H_1\cup H_2\cup H_3)\geq \sqrt{\frac{(1-\eps_0)(1-\eps')}{2}}$.
\item[(b)] Let $S_i$ be a generating set for $H_i$ and $\delta=\min\{\kappa(H_i,S_i)\}_{i=1}^3$.
Then  $$\kappa(G,S_1\cup S_2\cup S_3)\geq \frac{\delta}{2}\sqrt{(1-\eps_0)(1-\eps')}.$$
\item[(c)] Assume in addition that each pair $(G,H_i)$ has relative property $(T)$.
Then $G$ has property $(T)$.
\end{itemize}
\end{Corollary}
\begin{proof} As in Lemma~\ref{orthogT}, (a) and (c) directly follow from (b).
Part (b) is a combination of Proposition~\ref{3subgps}, Lemma~\ref{orthogT}(b) (applied twice),
and the fact that if subgroups $H$ and $K$ of some group are $\eps$-orthogonal, then
$\rho(H,K)\leq\frac{1+\eps}{2}$.
\end{proof}

Corollary~\ref{T3} has a straightforward generalization to the case of groups generated
by $n$ subgroups; however, we shall not prove or even state it. Instead, in Section~5
we proceed with a spectral criterion for property $(T)$ which will yield a stronger result
(see Corollary~\ref{Tn}).
\vskip .1cm
We finish this section with a simple lemma which will be used later:
\begin{Lemma}
\label{orth_normal} Let $G$ be a group generated by subgroups $H$ and $K$,
and suppose that $H$ is normal in $G$. Then $H$ and $K$ are $0$-orthogonal.
\end{Lemma}
\begin{proof} Let $V\in\Rep_0(G)$, take any $v\in V^K$ and let $w=P_{V^H}(v)$.
Since $H$ is normal, the subspaces $V^H$ and $(V^H)^{\perp}$ are $G$-invariant,
which implies that $w$ must be $K$-invariant. Therefore, $w\in V^H\cap V^K\subseteq V^{G}=\{0\}$,
whence $v\in (V^H)^{\perp}$.
\end{proof}

\section{Unitary representations of groups of nilpotency class two}

In this section we will consider the following problem: given a
group $G$ of nilpotency class $2$, generated by two abelian
subgroups $X$ and $Y$, we wish to compute (or estimate from above)
the orthogonality constant between $V^X$ and $V^Y$ where $V\in
\Rep_0(G)$. Our main results are as follows. First we will show
that if $V$ is a finite-dimensional representation of $G$, then
$V^X$ and $V^Y$ are $\frac{1}{\sqrt{\dim V}}$-orthogonal (see
Theorem~\ref{finitedimension}). Then we obtain a lower bound on
the degree of a non-one-dimensional irreducible unitary representation of $G$.
This result applies when $G$ is an $A$-group for some unital ring $A$, on
which the bound depends (see Theorem~\ref{mg}). Next we show that
if $V$ is a representation of $G$ without finite-dimensional
subrepresentations, then $V^X$ and $V^Y$ are orthogonal, under
additional assumptions on $G$ which hold whenever $G$ is finitely
generated (see Theorem~\ref{class2fg}). Combining these three
results, we obtain an upper bound for the orthogonality constant
$\eps(X,Y)$ which applies to Noetherian $A$-groups (see
Corollary~\ref{cor:main}). Finally, we show that $X$ and $Y$ are
$\frac{1}{\sqrt{2}}$-orthogonal for any group $G$ of nilpotency
class $2$, generated by two abelian subgroups $X$ and $Y$ (see
Proposition~\ref{class2:general}).

Before turning to general theory, we briefly discuss
the representation theory of the discrete Heisenberg group
which we hope will help the reader understand the overall
picture.

\subsection{Representations of the discrete Heisenberg group}

Let $R$ be an associative ring. Define $H(R)$ to be the group of $3\times 3$
upper-unitriangular matrices with entries in $R$. We will call $H(R)$ the
{\it Heisenberg group over $R$.} The group $H(\dbZ)$ is often
referred to as the {\it discrete Heisenberg group} and is given
by the presentation $\la x,y,z \mid [x,y]=z, [x,z]=[y,z]=1\ra$ where
\vskip .12cm
\centerline{\small $x=\left (\begin{array}{ccc}1 & 1& 0\\
0& 1& 0\\0&0 &1\end{array} \right)$,
$y=\left (\begin{array}{ccc}1 & 0& 0\\
0& 1& 1\\0&0 &1\end{array} \right)$ and
$z=\left (\begin{array}{ccc}1 & 0& 1\\
0& 1& 0\\0&0 &1\end{array} \right)$.}
Note that $H(\dbZ)$ is of nilpotency class two and generated
by abelian subgroups $X=\la x\ra$ and $Y=\la y\ra$.
\vskip .12cm

It is well-known (see Theorem~\ref{descr} below) that every
finite-dimensional irreducible representation of a finitely
generated nilpotent group factors through a finite quotient. Thus,
the study of finite-dimensional irreducible representations of $H(\dbZ)$ is
easily reduced to that of the groups $H(\dbZ/p^n\dbZ)$ where
$n\geq 1$ and $p$ is prime. We will consider the case $n=1$, that
is, describe representations of $H(\dbF_p)$ (the case $n>1$ is
more complex, but similar).

The group $H_p=H(\dbF_p)$ has $p^2$ (irreducible) representations of
degree 1 and $p-1$ irreducible representations of degree $p$ described as follows:
Let $e_1,e_2,\ldots, e_p$ be an orthonormal basis for $\dbC^p$, and let $\zeta$
be a $p^{\rm th}$ root of unity. Then we can define a representation $\rho_{\zeta}: H_p\to GL_p(\dbC)$
by setting $$\rho_{\zeta}(x)(e_i)=\zeta^i e_i \mbox{ and }\rho_{\zeta}(y)(e_i)=e_{i+1},$$
where indices are taken modulo $p$. It is easy to see that $\rho_{\zeta}$ is irreducible and unitary,
and the $p-1$ choices for $\zeta$ yield $p-1$ pairwise non-equivalent representations.

If $V$ is a non-trivial one-dimensional representation of $H_p$,
then either $V^{\la x\ra}=\{0\}$ or $V^{\la y\ra}=\{0\}$, so
$V^{\la x\ra}$ and $V^{\la y\ra}$ are orthogonal. If $V$ is one of
the above $p$-dimensional representations, then $V^{\la x\ra}=\dbC
e_1$ and $V^{\la y\ra}=\dbC f$ where $f=\sum\limits_{i=1}^p e_i$.
Since $\frac{\la e_1, f \ra}{\|e_1\|\cdot\|f\|}=
\frac{1}{\sqrt{p}}$, the subspaces $V^{\la x\ra}$ and $V^{\la
y\ra}$ are $\frac{1}{\sqrt{p}}$-orthogonal.

\vskip .1cm
We now turn to infinite-dimensional representations of $H=H(\dbZ)$. Fix
a separable infinite-dimensional Hilbert space $V$ with orthonormal basis
$\{e_k\}_{k\in\dbZ}$. For each $\lam\in\dbC$, with $|\lam|=1$,
define the unitary representation $\rho_{\lam}$ of $H$ on $V$ by setting
$$\rho_{\lam}(x) e_k=e_{k+1} \mbox{ and } \rho_{\lam}(y) e_k=\lam^k e_k.$$
For any $\lam\neq 1$, we have $V^{\la y \ra}=e_0$ and $V^{\la x\ra}=\{0\}$, so
$V^{\la x\ra}$ and $V^{\la y\ra}$ are orthogonal.

One can show that $\rho_{\lam}$ is irreducible provided $\lam$ is not a root of unity.
Furthermore, any irreducible representation of $H$ which has a $y$-invariant vector must be
(unitarily) equivalent to $\rho_{\lam}$ for some $\lam$.
It might be possible to use this fact and decompositions of unitary
representations into direct integral of irreducibles to provide
alternative proofs of some of the results in this section in the case $G=H(\dbZ)$.

\subsection {Some auxiliary results} The following two results
will be used in our analysis of representations of groups
of nilpotency class two. The first one, due to
Lubotzky and Magid~\cite{LuMa}, reduces the
study of finite-dimensional irreducible complex representations
of a finitely generated nilpotent group to those with finite image.

\begin{Proposition}\label{descr}\cite[Theorem 6.6]{LuMa} Let $\Gamma$ be a
finitely generated nilpotent group. Then for each irreducible
representation $\rho:\Gamma\to \GL_n(\dbC)$ there exists a linear
representation $\lambda:\Gamma\to \dbC^*$ and an irreducible
representation $\sigma:\Gamma\to \GL_n(\dbC)$ with finite image
such that $\rho=\lambda\otimes \sigma$.
\end{Proposition}

It is clear that if in Proposition~\ref{descr} $\rho$ is unitary,
then $\lambda$ should also be unitary.

The second result uses the notion of a convergent  net.  We
include the relevant definitions  for the convenience of the
reader. A directed set $\grA$ is a partially ordered set with the
property that for each $\alpha,\beta\in \grA$ there
exists $\gamma\in \grA$ such that $\gamma\ge \alpha$ and $\gamma
\ge \beta$. A net on a set $X$ is a function $\alpha\mapsto \lambda_\alpha$
from some directed set $\grA$ to $X$. If $X$ is a topological space,
a net $\{\lambda_\alpha\}_{\alpha\in \grA}$ on $X$ is said to converge to $\lambda\in X$
if for each neighborhood $U$ of $\lambda$ there exists $\alpha_U\in \grA$
such that $\lambda_\alpha\in U$ for all $\alpha\ge \alpha_U$.

\begin{Lemma}\label{nested} Let $\grA$ be a directed set, $V$ a Hilbert space and $\{U_\alpha\}_{\alpha\in A}$
a set of subspaces of $V$ such that $U_\alpha \subseteq U_{\beta}$
if  $\alpha\le \beta$. Assume that $\overline{\cup U_\alpha}=V$.
Then for any $v\in V$ the net $\{P_{U_\alpha}(v)\}_{\alpha\in \grA}$
converges to $v$.
\end{Lemma}
\begin{proof}
By \cite[Proposition 4.64]{Dou}, $\{P_{U_\alpha}(v)\}_{\alpha\in \grA}$ converges to some $u\in V$.
Since $P_{U_\alpha}(P_{U_\beta}(v))=P_{U_\alpha}(v)$ for all
$\beta\ge \alpha$, $P_{U_\alpha}(u)=P_{U_{\alpha}}(v)$ for all
$\alpha\in \grA$.  Hence $v-u\in \cap U_\alpha^\perp =\{0\}$, and so $u=v$.
\end{proof}

An important special case of Lemma~\ref{nested} is when $\grA=\dbN$ (natural numbers)
and $\{U_\alpha\}_{\alpha\in A}=\{U_n\}_{n=1}^{\infty}$ is an ascending chain
of subspaces whose union is $V$. If $V$ is a separable Hilbert space, the general
form of Lemma~\ref{nested} can be easily reduced to this special case.

\subsection {Representations of groups of nilpotency class two}

The following notations will be fixed throughout this subsection.
By $G$ we always denote a  group of nilpotency class two generated
by two abelian subgroups $X$ and $Y$, and we let $Z=[G,G]$. If $V$
is a $G$-module and $v\in V$, we set $C_Z(v)=\{z\in Z : zv = v\}$
and $X(v)=\{x\in X : [x,Y]\le C_Z(v)\}$. If $H$ is a subset of $G$,
by an $H$-subspace of $V$ we mean an $H$-invariant subspace.
We start with the
following technical lemma.
\begin{Lemma}\label{maxcentr}
Let $V\in\Rep(G)$, let $U$ be a $Z$-subspace of
$V^Y$ and $0\ne u\in U$ be such that $X(u)$ is maximal among
$\{X(v) : 0\ne v\in U\}$. Then for any $x  \in X\setminus X(u)$
we have $xu\in (V^Y)^\perp$.
\end{Lemma}
\begin{proof} Let $y\in Y$ and $v\in V^Y$. Since $u\in V^Y$, we have
\begin{equation}
\la xu,v\ra=\la xyu,yv\ra=\la y^{-1}xyu,v\ra=\la x[x,y]u,v\ra. \label{eq:xy}
\end{equation}
Fix $x\in X\setminus X(u)$.
Let $M$ be the (closed) subspace spanned by the subset $\{[x,y]u : y\in Y\}\subseteq U$.
Note that for any $g\in G$ we have
$C_Z(gu)=C_Z(u)$, and thus $X(gu)=X(u)$. Since $X(u)$ is maximal,
we have $X(v)=X(u)$ for any $0\neq v\in M$. Thus $M^{[x,Y]}=\{0\}$
as $x\not\in X(u)$.

Let $W$ be the subspace spanned by $\{([x,y]-1)u : y\in Y\}$.
Equation \eqref{eq:xy} implies that $xw\in (V^Y)^\perp$
for any $w\in W$, so we only need to show that $u\in W$. We claim
that in fact $W=M$. Let us  show that $W^{\perp M}=\{0\}$.

Let $v\in  W^{\perp M}$ and $y\in Y$. Note that $([x,y]-1)v\in W$
since $v$ can be approximated by finite sums $\sum[x,y_i]u$, with $y_i\in Y$,
and $$([x,y]-1)[x,y_i]u=([x,yy_i]-1)u-([x,y_i]-1)u\in W.$$
Similarly, one shows that $W$ is $[x,y']$-invariant for any $y'\in Y$,
and since $[x,y]^{-1}=[x,y^{-1}]$, we get
$([x,y]^{-1}-1)([x,y]-1)v\in W$. Therefore,
$$\|([x,y]-1)v\|^2=\la v,([x,y]^{-1}-1)([x,y]-1)v\ra=0 \mbox{ as }v\in W^{\perp M}.$$
Hence, $([x,y]-1)v=0$ and so $v\in M^{[x,Y]}$. Therefore $v=0$.
\end{proof}

\begin{Theorem}\label{finitedimension} Let $\rho:G\to GL(V)$ be an irreducible
finite-dimensional unitary representation of $G$, and let $n=\dim V$.
Then $V^X$ and $V^Y$ are orthogonal if $V$ is non-trivial of dimension
$1$ and $\frac 1{\sqrt n}$-orthogonal in general.
\end{Theorem}
\begin{proof}
The case $\dim V=1$ is obvious. Consider the case $n=\dim V >1$.

First let us assume that $G$ is finitely generated. By Proposition
\ref{descr} and the remark after it, there exists a unitary
representation $\lambda:\Gamma\to \dbC^*$ and an irreducible
representation $\sigma:\Gamma\to \GL_n(\dbC)$ with finite image
such that $\rho=\lambda\otimes \sigma$.

Let $m=|\sigma(G)|$. If $g\in G$ has a nonzero fixed vector in $V$, then
$\lambda(g)^m=1$. Without loss of generality we may assume that
$V^X$ and $V^Y$ are non-trivial, whence $\rho(G)$ should be
finite.

Changing $G$ by $\rho(G)$, if necessary, we may assume that $\rho$
is faithful. Then by Schur's lemma each element of $Z(G)\backslash\{1\}$ acts as a scalar
$\mu\neq 1$. Since $V^X$ is not trivial, $X\cap Z(G)=\{1\}$ and so
$X\cap C_G(Y)=\{1\}$. Thus, $[x,Y]\ne \{1\}$ for any $x\in X\setminus
\{1\}$.

Since $C_Z(v)=\{1\}$ for any nonzero $v\in V$, we can apply
Lemma~\ref{maxcentr} with $U=V^Y$, where we can let $u$ be any
nonzero element of $U$ and $x$ be any non-identity element of $X$.
It follows that elements of the set $\{xu : x\in X\setminus
\{1\}\}$ are pairwise orthogonal (where $0\neq u$ is some fixed
element of $U$). Since $\dbC u$ is $\la Y,Z\ra$-invariant, the
$\dbC$-span of $\{xu : x\in X\setminus \{1\}\}$ must be
$G$-invariant (and thus equal to $V$). Thus, $\{xu : x\in X\}$
is in fact an orthogonal basis of $V$. In particular, $\dim V=|X|$
and $V^X=\dbC\sum_{x\in X}xu$, whence $\dim V^X=1$. By symmetry, we
have $\dim V^Y=1$, and thus $V^Y=\dbC u$. Since,
$$\la u,\sum_{x\in X}xu \ra=\la u,u\ra=\frac 1{\sqrt {\dim V}}\| u\|\cdot\|\sum_{x\in X}xu\|,$$
we obtain that $V^X$ and $V^Y$ are $\frac 1{\sqrt {\dim V}}$-orthogonal.

Finally, we consider the general case. Since $V$ is
finite-dimensional, there exists a finitely generated subgroup
$\Gamma$ of $G$ such that $\Gamma$ is generated by $X_1=\Gamma\cap
X$ and $Y_1=\Gamma\cap Y$ and  $V$ is also irreducible for
$\Gamma$. Thus, $V^{X_1}$ and $V^{Y_1}$ are $\frac 1{\sqrt
n}$-orthogonal and, in particular, $V^X$ and $V^Y$ are $\frac
1{\sqrt n}$-orthogonal.
\end{proof}

>From now on we will denote by $m(G)$ the smallest degree of a
non-one-dimensional irreducible unitary representation of $G$.

\begin{Definition}\rm
\label{S-group}
Let $A$ be an associative ring with $1$. We will say
that $G$ has the structure of an {\it $A$-group} if
\begin{itemize}
\item[(i)] $X$ is a right $A$-module and $Y$ is a left
$A$-module. \item[(ii)] For any $r\in A$, $x\in X$ and $y\in Y$ we
have $[xr,y]=[x,ry]$.
\end{itemize}
We will say that $G$ is a Noetherian $A$-group if both $X$ and $Y$ are Noetherian $A$-modules.
\end{Definition}
Note that $G$ is always a $\dbZ$-group, and $G$ is a Noetherian $\dbZ$-group if and only if
$G$ is finitely generated. If $A$ is an arbitrary
ring with $1$, the Heisenberg group $H(A)$ is the simplest example
of a (Noetherian) $A$-group.

\begin{Proposition}\label{mg}
Assume that $G$ is an $A$-group for some ring $A$.
Then $m(G)$ is at least the minimal index of a proper $A$-submodule of $X$.
\end{Proposition}
\begin{proof}
Let $V$ be an irreducible finite-dimensional unitary
representation of $G$. Since $\la Z,Y\ra$ is abelian, there exists
a $\la Z,Y\ra $-eigenvector $v$. If $X(v)=X$, then all elements
from $\la Z,Y \ra$ act as scalars. Hence, $[G,G]$ is in the
kernel of the representation, and so $V$ is one-dimensional.

If $x\in X\setminus X(v)$ then $xv$ is also a $\la Z,Y\ra
$-eigenvector, but corresponding to a different character.
Furthermore, the characters corresponding to two distinct elements
$x,x'\in X$ are the same if and only if $x'x^{-1}\in X(v)$. Hence
$\dim V\ge |X:X(v)|$. Finally, since $[xr,y]=[x,ry]$ for any
$r\in A$, $x\in X$ and $y\in Y$, we obtain that for any $v\in V$
the set $X(v)$ is an $A$-submodule of $X$.
\end{proof}

\begin{Theorem}\label{class2fg}
Assume that $G$ has the structure of a Noetherian $A$-group for some ring $A$.
Let $V\in\Rep(G)$, and assume that $V$ has no finite-dimensional $G$-subspaces.
Then $V^{X}$ and $V^{Y}$ are orthogonal.
\end{Theorem}
\begin{Remark}The hypotheses of the theorem clearly hold
whenever $G$ is finitely generated (in which case we take $A=\dbZ$).
\end{Remark}
\begin{proof}
Let $V_{\rm sm}$ be the subspace of $V$ generated by all vectors $v\in V$ such that $|X:X(v)|< \infty$.
\footnote{Recall that by `subspace' we mean a closed subspace. The set of all $v\in V$ such that $|X:X(v)|< \infty$
is an ``abstract subspace'' of $V$ (not necessarily closed).}
Note that $V_{\rm sm}$ is $G$-invariant, and thus $V$ is a direct sum of subrepresentations
$V_{\rm sm}$ and $V_{\rm sm}^{\perp}$.
Since Theorem~\ref{class2fg} holds for a direct sum if and only if it holds for each direct
summand, it is enough to consider two cases: $V_{\rm sm}=V$ and
$V_{\rm sm}=\{0\}$.

{\it Case 1: $V_{\rm sm}=V$. } In this case $V$ has a generating
set $ \{v_i\}_{i\in I}$ consisting of vectors satisfying
$|X:X(v_i)|<\infty$. Denote by $\grA$ the set of all finite
subsets of $I$. Then $\grA$ is a directed set with respect to
inclusion. For any $\alpha\in \grA$, let $U_\alpha$ be the
$G$-subspace generated by $\{v_i\}_{i\in \alpha}$.
Lemma~\ref{nested} implies that $\{P_{U_\alpha}(v)\}_{\alpha\in\grA}$ converges to $v$
for any $v\in V$, so it is enough to show that $P_{U_\alpha}(V^X)=U_{\alpha}^X$ is orthogonal
to $P_{U_\alpha}(V^Y)=U_{\alpha}^Y$ for each $\alpha\in \grA$. Thus, from now on we may assume that
$V$ is generated as a $G$-subspace by a finite set $\{v_i : 1\leq i\leq l\}$.

Let $H=\cap_{i=1}^l X(v_i)$. Then $H$ is a finite index subgroup of $X$
and $[H,Y]$ acts trivially on $V$, so $V^Y$ is $H$-invariant. Let $w\in V^X$ and $w_0\in V^Y$
be the projection of $w$ to $V^Y$. Then $w_0$ is $\la H,Y \ra$-invariant.

Recall that $Y$ is a Noetherian left $A$-module, and let
$\{y_1,\ldots, y_n\}\subset Y$ be a finite generating set for $Y$ as an $A$-module.
%such that $Y=Ay_1+\cdots+Ay_n$.
Since $[x,ry]=[xr,y]$ for any $x\in X$, $y\in Y$ and $r\in A$, we have
$Z=[X,Y]=[X,y_1]\cdots [X,y_n]$, and so $[H,Y]\supseteq [H,y_1]\cdots
[H,y_n]$ is of finite index in $Z$. Note that $G=XYZ$. Thus $\la
H,Y \ra\supseteq HY[H,Y]$ is a finite index subgroup of $G$, and so $w_0$ generates a
$G$-subspace of finite dimension. Hence $w_0=0$ and so $w\in
(V^Y)^\perp$.

{\it Case 2: $V_{\rm sm}=\{0\}$. } Let  $0\ne w\in V^X$ and
$$U=P_{V^Y}(\dbC[Z]w)=(\overline{\dbC[Z]w + (V^Y)^{\perp}})\cap V^Y.$$
Note that $U$ is $\la Y,Z\ra$-invariant. We want to show that
$U=\{0\}$. Assume $U\ne \{0\}$. Among all $0\ne u\in U$ we choose such $u$
for which $X(u)$ is maximal (we may do this because $X$ is
Noetherian). Note that $X(u)$ is of infinite index in $X$ since
$V_{\rm sm}=\{0\}$.

Let $w_0=P_{\dbC[G]u}(w)$ be the projection of $w$ onto the $G$-subspace
generated by $u$. Note that $w_0\in V^X$. Since $w_0\in \dbC[G]u$,
$\la Y,Z\ra u\subseteq U$ and $X(gu)=X(u)$ for any $g\in G$, we
may approximate $w_0$ by finite sums of the
form $\sum_{i=1}^l x_iu_i$, where $x_i\in X$ and $u_i\in U$
satisfy $X(u_i)=X(u)$.

Let $T$ be a transversal of $X(u)$ in $X$. Note that if $a,b\in X$,
then for  all but  one $t\in T$, we have $tab^{-1}\not \in  X(u)$. Hence, by Lemma \ref{maxcentr},
$\la t\sum_{i=1}^l x_iu_i,\sum_{i=1}^l x_iu_i \ra=0$ for almost all
$t\in T$. Since $T$ is infinite and $tw_0=w_0$ for all $t\in T$,
it follows that $\la w_0,w_0\ra=0$, whence $w_0=0$. Thus $w\in
(\dbC[G]u)^\perp$ and so $u\in (\dbC[G]w)^\perp\cap V^Y\subseteq (\dbC[Z]w)^\perp\cap V^Y$. This
implies that $u=0$, a contradiction. Thus, $U=\{0\}$ and, in
particular, $w\in (V^Y)^\perp$.
\end{proof}

\begin{Corollary}
\label{cor:main} Let $G$ be a  group of nilpotency class two
generated by abelian subgroups $X$ and $Y$. Assume that there
exists a ring $A$ such that $G$ is a Noetherian $A$-group. Let $m$
be the smallest index of a proper $A$-submodule of $X$. Then the
subgroups $X$ and $Y$ are $\frac{1}{\sqrt{m}}$-orthogonal.
\end{Corollary}
\begin{Remark}
We do not know whether Corollary~\ref{cor:main} holds without the
hypothesis `$X$ and $Y$ are Noetherian $A$-modules'.
\end{Remark}
\begin{proof}
Let $W\in\Rep(G)$ such that $W=\oplus W_{\alpha}$ for some
family of representations $\{W_{\alpha}\}$. Then by Cauchy
inequality $\eps(W^X,W^Y)=\sup\limits_{\alpha}\eps(W_{\alpha}^X,W_{\alpha}^Y)$.
Every $V\in\Rep_0(G)$ can be written as $V=V_1\oplus V_2$ where
$V_1$ is a sum of non-trivial finite-dimensional irreducible
representations and $V_2$ has no finite-dimensional subrepresentations.
Thus, Corollary~\ref{cor:main} follows from Theorem~\ref{finitedimension}, Proposition~\ref{mg}
and Theorem~\ref{class2fg}.
\end{proof}
Finally, we obtain a `universal bound' for the orthogonality constant
between $V^X$ and $V^Y$ which holds without any additional assumptions on the group $G$
or unitary representation $V$:
\begin{Proposition}\label{class2:general} Let $G$ be a group of nilpotency class two generated by abelian subgroups $X$
and $Y$. Then $X$ and $Y$ are $\frac{1}{\sqrt 2}$-orthogonal.
\end{Proposition}
\begin{proof}
Let $\{x_i\}_{i\in I}$ and $\{y_j\}_{j\in J}$ be generating sets
for $X$ and $Y$, respectively. Let $\grA$ be the set of pairs
$\alpha=(\alpha_1,\alpha_2)$, where $\alpha_1$ is a finite subset
of $I$ and $\alpha_2$ is a finite subset of $J$. For each
$\alpha=(\alpha_1,\alpha_2)\in \grA$ let $X_{\alpha}=\langle x_i :
i\in \alpha_1\rangle$, $Y_\alpha=\langle y_j :
j\in\alpha_2\rangle$ and $G_\alpha=\langle X_\alpha,
Y_\alpha\rangle$, and put $U_\alpha=(V^{G_\alpha})^\perp$. By
Corollary~\ref{cor:main} applied to the group $G_\alpha$ with
$A=\Z$, the spaces $P_{U_\alpha}(V^X)$ and $P_{U_\alpha}(V^Y)$ are
$\frac 1{\sqrt 2}$-orthogonal. Lemma \ref{nested} applied to
$\{U_\alpha\}_{\alpha\in \grA}$ implies that $V^X$ and $V^Y$ are
$\frac 1{\sqrt 2}$-orthogonal.
\end{proof}

\section{The main criterion for property $(T)$}

In this section we consider groups associated with graphs of
groups and show that such a group has property $(T)$ provided
certain ``weighted Laplacian'' of the underlying graph has large
first eigenvalue -- see Theorem~\ref{GG1} (basic version) and
Theorem~\ref{GG2} (weighted version). As a straightforward
application of Theorem~\ref{GG1} we obtain a generalization of
Corollary~\ref{T3} to the case of groups generated by $n$
subgroups which are pairwise $\eps$-orthogonal for some small
$\eps$, while the more technical Theorem~\ref{GG2} can be used to
recover and improve the full statement of Corollary~\ref{T3}.

We remark that there is a well-known criterion for property $(T)$
for groups defined by generators and relations due to
\.Zuk~\cite{Zu}. The setting in \.Zuk's criterion is different
from ours, although its statement also involves the first
eigenvalue of certain Laplacian. We do not know if there exists a
``deep connection'' between \.Zuk's criterion and
Theorem~\ref{GG1}. In any case, the two spectral criteria seem to
be applicable to different kinds of groups.

\subsection{Preliminaries}
Let $Y$ be a finite  graph without loops. For any edge $e=(x,y)\in
\Edg(Y)$, we denote by $\bar e=(y,x)$ the inverse of $e$. We assume
that if $e\in \Edg(Y)$, then also $\bar e \in \Edg(Y)$. If $e=(x,y)$, we
denote by $e^+=y$  the head of $e$ and by $e^-= x $ the tail of
$e$. If $y\in \Vert(Y)$, then $\deg (y)$ denotes the degree of $y$:
$$\deg (y)=|\{e\in \Edg(Y):y=e^+\}|.$$

{\bf Graphs of groups.} A graph of groups $\Yhat$ over $Y$ is an
assignment of a group $G_y$ to each vertex $y\in \Vert(Y)$ and a group
$G_e$ to each edge $e\in \Edg(Y)$, as well as injective homomorphisms
$\phi_{(e,-)}:G_{e}\to G_{e^-}$ and $\phi_{(e,+)}:G_{e}\to
G_{e^+}$ for each $e\in \Edg(Y)$. We will assume that $G_e=G_{\bar
e}$ and $\phi_{(e,-)}=\phi_{(\bar e,+)}$.

Let $G(\Yhat)$ be the group generated by (isomorphic copies) of
vertex groups $\{G_y : y\in \Vert(Y)\}$ subject to relations
$$\phi_{(e,-)}(g)=\phi_{(e,+)}(g) \mbox{ for any } e\in \Edg(Y) \mbox{ and }g\in G_e.$$
It is common to say that the group $G(\Yhat)$ is {\it associated
with the graph of groups $\Yhat$. } The following terminology is
non-standard, but convenient for our purposes:
\begin{Definition}\rm
Let $G$ be a group and $Y$ a finite graph without loops. A {\it
decomposition of $G$} over $Y$ is a choice of a vertex subgroup
$G_y\subseteq G$ for any $y\in \Vert(Y)$ and an edge subgroup
$G_e\subseteq G$ for any $e\in \Edg(Y)$ such that
\begin{itemize}
\item[(a)] The vertex subgroups $\{G_y : y\in \Vert(Y)\}$ generate
$G$; \item[(b)] $G_e=G_{\ebar} \mbox{ and } G_e\subseteq
G_{e^+}\cap G_{e^-}\mbox{ for any } e\in \Edg(Y).$
\end{itemize}
\end{Definition}
It is clear that each decomposition of $G$ over $Y$ corresponds to
a presentation of $G$ as a quotient of the group $G(\Yhat)$
associated with some graph of groups $\Yhat$ over $Y$.

{\bf Laplacians.} Let $Y$ be a finite connected graph without
loops. Fix two functions $\alpha:\Vert(Y)\to \R_{>0}$ and $c:\Edg(Y)\to
\R_{>0}$, and let $V$ be a Hilbert space. Let $\Omega^0(Y)$ be the
Hilbert space of functions $f:\Vert(Y)\to V$ with the scalar product
\begin{equation}
\label{omega0} \la f,g \ra=\sum_{y\in \Vert(Y)}\frac{\la f(y),g(y)
\ra}{\alpha(y)}
\end{equation} and let
$\Omega^1(Y)$ be the Hilbert space of functions $f:\Edg(Y)\to V$ with
the scalar product
\begin{equation}
\label{omega1} \la f,g \ra= \sum_{e\in \Edg(Y)}\la f(e),g(e)\ra c(e).
\end{equation}
Define the linear operator
$$d:\Omega^0(Y)\to \Omega^1(Y) \mbox{ by } (df)(e)=\frac{1}{c(e)+c(\bar e)}(f(e^+)-f(e^-)).$$
Then the adjoint operator $d^*:\Omega^1(Y)\to \Omega^0(Y)$ is
given by formula
$$(d^*f)(y)=\alpha(y) \sum_{y=e^+}\frac{1}{c(e)+c(\bar e)}\left(c(e)f(e)-c(\bar e)f(\bar e)\right).$$
Define $\Delta=d^*d:\Omega^0(Y)\to \Omega^0(Y)$. Then
$$(\Delta f)(y)=\alpha(y)\sum_{y=e^+}\frac{1}{c(e)+c(\bar e)}(f(y)-f(e^-))=\alpha(y)\sum_{y=e^+} df(e).$$
We will refer to $\Delta$ as the {\it weighted Laplacian of $Y$}
corresponding to the weight functions $\alpha$ and $c$.

Note that if $\alpha(y)= 1$ and $c(e)=\frac{1}{2}$ for any $y\in
\Vert(Y)$ and $e\in \Edg(Y)$, then $\Delta$ is the standard Laplacian of
$Y$:
$$(\Delta f)(y)=\deg(y)f(y)-\sum_{y=e^+}f(e^-) \mbox{ for } f\in
\Omega^0(Y).$$ This is the Laplacian that will be used for the
basic version of our criterion.

As usual, by $\lam_1(\Delta)$ we denote the smallest positive eigenvalue of
$\Delta$. More generally, for an arbitrary non-negative self-adjoint operator
$A:Z\to Z$ (where $Z$ is some Hilbert space) we define $\lam_1(A)$
to be the minimum of the spectrum of the restriction of $A$ to $(\Ker A)^{\perp Z}$.
Thus
\begin{equation}
\label{def:lam1}
\lam_1(A)=\inf\limits_{0\ne v\in (\Ker A)^{\perp Z}}\frac{\la Av,v\ra}{\|v\|^2}.
\end{equation}

\subsection{Spectral criterion (basic version): statement and examples.}
\begin{Theorem} \label{GG1}
Let $Y$ be a finite connected $k$-regular graph for some $k\geq
2$, and let $\Delta$ be the standard Laplacian of $Y$, that is,
$$\Delta(f)(y)= \sum_{y=e^+}(f(y)-f(e^-))=kf(y)-\sum_{y=e^+}f(e^-).$$
Let $G$ be a group and $(\{G_y: y\in \Vert(Y)\},\{G_e: e\in \Edg(Y)\})$ a
decomposition of $G$ over $Y$. For each $y\in \Vert(Y)$, we set
$\rho(y)=\rho(\{G_e : y=e^+\},G_y)$ and
$$\rho=\max \{\rho(y) : y\in \Vert(Y)\}.$$
Then
$$\rho(\{G_y : y\in \Vert(Y)\})\le\frac{\rho}{1-\rho}\left( \frac{2k}{\lambda_1(\Delta)}-1\right).$$
In particular, if $\rho<\frac {\lambda_1(\Delta)}{2k}$, the
Kazhdan constant $\kappa(G,\bigcup\limits_{y\in \Vert(Y)}G_y)$ is
positive.
\end{Theorem}
\begin{Remark} Note that $\rho(\{G_y : y\in \Vert(Y)\})=\rho(\{G_y : y\in \Vert(Y)\}, G)$
since $G$ is generated by $\{G_y : y\in \Vert(Y)\}$ by the definition of decomposition
of $G$ over $Y$.
\end{Remark}

\begin{Example}
\label{example:complete} \rm Assume that $Y$ is a complete graph
on $n$ vertices. Then
$$(\Delta f)(y)= (n-1)f(y)-\sum_{z\neq y}f(z).$$
Hence $\lambda_1(\Delta)=n$, and so $\rho(\{G_y\})\le
\frac{\rho(n-2)}{(1-\rho)n}.$
\end{Example}

Note that Corollary~\ref{T3} in the case
$\eps_1=\eps_2=\eps_3<\frac{1}{2}$ is a special case of
Example~\ref{example:complete}. Indeed, suppose that a group $G$
is generated by three subgroups $H_1, H_2, H_3$. Then $G$
naturally decomposes over the complete graph $Y$ on three vertices
$\{1,2,3\}$, with edge groups $$G_{(1,2)}=G_{(2,1)}=H_3,\quad
G_{(2,3)}=G_{(3,2)}=H_1, \quad G_{(3,1)}=G_{(1,3)}=H_2$$ and
vertex groups
$$G_{1}=\la H_2,H_3\ra, \quad G_{2}=\la H_1,H_3\ra, \quad G_{3}=\la H_1,H_2\ra.$$
If $H_1, H_2, H_3$ are pairwise $\eps$-orthogonal, then $\rho\leq
\frac{1+\eps}{2}$, so $\rho(\{G_{1},G_{2},G_{3}\})\leq
\frac{(1+\eps)}{3(1-\eps)}$, which is less than $1$ if and only if
$\eps<1/2$.

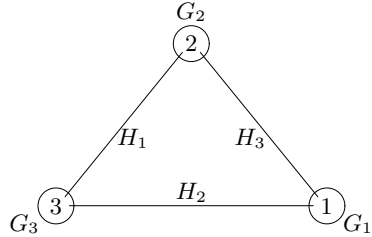
\begin{picture}(100,100)(40,40)
\setlength{\unitlength}{0.06cm} \drawcircle{70.0}{32.0}{8.0}{}
\drawcircle{130.0}{32.0}{8.0}{}\drawcircle{100.0}{68.0}{8.0}{}
\drawpath{73.0}{32.0}{127.0}{32.0}
\drawpath{72.0}{34.0}{98.0}{66.0}\drawpath{102.0}{66.0}{128.0}{34.0}
\drawcenteredtext{70.0}{32.0}{\footnotesize$3$}
\drawcenteredtext{100.0}{68.0}{\footnotesize$2$}
\drawcenteredtext{130.0}{32.0}{\footnotesize$1$}
\drawcenteredtext{63.0}{28.0}{\footnotesize $G_{3}$}
\drawcenteredtext{137.0}{28.0}{\footnotesize $G_{1}$}
\drawcenteredtext{100.0}{75.0}{\footnotesize $G_{2}$}
\drawcenteredtext{100.0}{35.0}{\footnotesize $H_{2}$}
\drawcenteredtext{87.0}{47.0}{\footnotesize $H_{1}$}
\drawcenteredtext{113.0}{47.0}{\footnotesize $H_{3}$}
\end{picture}

We shall now extend this argument to the case of groups generated
by $n$ subgroups, which are pairwise $\eps$-orthogonal for small
$\eps$:

\begin{Corollary}
\label{Tn} Let $G$ be a group, and let $H_1,\ldots, H_n$ (where
$n\geq 2$) be subgroups of $G$ such that $G=\la H_1,\ldots,
H_n\ra$. Let $\eps=\max\{\eps(H_i, H_j) : i\neq j\}$, and assume
that $\eps<\frac 1 {n-1}$. The following hold:
\begin{itemize}
\item[(a)] $\kappa(G,\cup_{i=1}^n H_i)\geq
\sqrt{\frac{2(1-(n-1)\eps)}{n}}.$ \item[(b)] Let $S_i$ be a
generating set for $H_i$, and let
$\delta=\min\{\kappa(H_i,S_i)\}_{i=1}^n$. Then
$$\kappa(G,\cup_{i=1}^n S_i)\geq \delta\sqrt{\frac{1-(n-1)\eps}{n}}.$$
\item[(c)] Assume in addition that each pair $(G,H_i)$ has
relative property $(T)$. Then $G$ has property $(T)$.
\end{itemize}
\end{Corollary}
\begin{proof} As in Lemma~\ref{orthogT} and Corollary~\ref{T3}, we only need to prove (b).
Let $I$ be a subset of $\{1,\ldots, n\}$. Denote by $H_I$ the
subgroup of $G$ generated by $\{H_i: i\in I\}$ and let $F_I=\{H_J:
J\subset I,\,|J|=|I|-1\}$ . We will prove the following two
statements for any subset $I$ with $|I|\geq 2$ by induction on
$|I|$:
\begin{align*}
&\mbox{ (i) } \rho(F_I)\leq
\frc{1+\eps}{(1-(|I|-2)\eps)}\cdot\frac{1}{|I|}& &\mbox{ (ii)
}\kappa(H_I,\bigcup_{i\in I} S_i)\geq
\delta\sqrt{\frc{1-(|I|-1)\eps}{|I|}}&
\end{align*}
Note that (ii) in the case $|I|=n$ is precisely the statement of
the Corollary. \vskip .12cm If $|I|=2$, that is, $I=\{i,j\}$ for
some $i,j$, then by assumption
$\rho(F_{i,j})=\frac{1+\eps(H_i,H_j)}{2}\leq \frac{1+\eps}{2}$, so
(i) holds. By Lemma~\ref{orthogT} we have $\kappa(H_{i,j})\geq
\delta\sqrt{1-\frac{1+\eps}{2}}$, so (ii) holds.

Take any $m\geq 2$, assume that (i) holds when $|I|=m$, and take
any subset $I$ with $|I|=m+1$. Consider the complete graph on the
set $I$. To each vertex $i\in I$ we assign the group
$H_{I\setminus i}$ and to each edge $(i,j)$ we assign the subgroup
$H_{I\setminus \{i,j\}}$. Then by induction assumption, for any
$i\in I$ we have $\rho(F_{ I\setminus i})\leq \rho_m$ where
$$\rho_m=\frac{1+\eps}{m((1-(m-2)\eps)}.$$ Hence from
Example~\ref{example:complete} we obtain
\begin{multline*}
\rho(F_I)\leq \frac{\rho_m}{1-\rho_m}\cdot\frac{m-1}{m+1}=
\frac{1+\eps}{m(1-(m-2)\eps)-(1+\eps)}\cdot\frac{m-1}{m+1}= \\
\frac{1+\eps}{(m-1)-(m-1)^2\eps}\cdot\frac{m-1}{m+1}=\frac{1+\eps}{(m+1)(1-(m-1)\eps)}.
\end{multline*}
Thus we proved (i). \vskip .1cm

Now assume that (ii) holds when $|I|=m$, and take any subset $I$
with $|I|=m+1$. By induction assumption, for any $i\in I$ we have
$$\kappa(H_{I\setminus{i}},\bigcup_{j\in I\setminus{i}} S_j)\geq
\delta\sqrt{\frac{1-(m-1)\eps}{m}}.$$ Applying Lemma~\ref{orthogT}
to the collection of subgroups $F_I=\{H_{I\setminus{i}} : i\in
I\}$ and their generating sets $\{\bigcup_{j\in I\setminus{i}}S_j
: i\in I\}$ and using (i) for $I$, we get
\begin{multline*}
\kappa(H_{I},\bigcup_{i\in I} S_i)\geq
\delta\sqrt{\frac{1-(m-1)\eps}{m}}\sqrt{1-\frac{1+\eps}{(m+1)(1-(m-1)\eps)}}=\\
\delta\sqrt{\frac{1-(m-1)\eps}{m}-\frac{1+\eps}{m(m+1)}}=\delta\sqrt{\frac{1-m\eps}{m+1}}.
\end{multline*}
This proves (ii).
\end{proof}

\subsection{Proof of the basic version of the spectral criterion.}

\begin{proof}[Proof of Theorem~\ref{GG1}]
Let $V\in\Rep_0(G)$. Let $\Omega^0(Y)$ and $\Omega^1(Y)$ be
defined by \eqref{omega0} and \eqref{omega1}, respectively, with
$\alpha(y)=1$ and $c(e)=1/2$, and consider the following two
subspaces of $\Omega^0(Y)$:
\begin{align*}
&W=\{f: \Vert(Y)\to V : f(y)\in V^{G_y}\mbox{ for all }y\in \Vert(Y)\}\mbox{ and }&\\
&U=\{f: \Vert(Y)\to V : f \mbox{ is constant} \}; \mbox{ note that }
U=\Ker d=\Ker\Delta.&
\end{align*}
Thus, $\rho(\{V^{G_y} : y\in \Vert(Y)\})=(\eps(W,U))^2$, so our goal
is to show that $$(\eps(W,U))^2\leq
\frac{\rho}{1-\rho}\left(\frac{2k}{\lambda_1(\Delta)}-1\right).$$

Consider the subspace $V^\prime=U+W$. The following lemma will
play a key role in subsequent computations.
\begin{Lemma}
\label{Vprime} Let $h\in V^\prime$ and  $e\in \Edg(Y)$. Then
\begin{itemize}
\item[(a)] $(d h)(e)=h(e^+)-h(e^-)\in V^{G_e}.$ \item[(b)]
$P_{(V^{G_{e^+}})^{\perp}}(dh(e))\in V^{G_e}$
\end{itemize}
\end{Lemma}
\begin{proof}
(a) holds for $h\in U$, in which case $h(e^+)-h(e^-)=0$, and also
for $h\in W$, in which case $h(e^+)-h(e^-)\in
V^{G_{e^+}}+V^{G_{e^-}} \subseteq V^{G_{e^+}\cap G_{e^-}}\subseteq
V^{G_e}$. By linearity (a) holds for any $h\in V^\prime$. Since
$$P_{(V^{G_{e^+}})^{\perp}}(dh(e))=dh(e)-P_{V^{G_{e^+}}}(dh(e))$$
and $P_{V^{G_{e^+}}}(dh(e))\in V^{G_{e^+}}\subseteq V^{G_e}$, (b)
follows from (a).
\end{proof}

Now let $U_1=U^{\perp V^\prime}$ and $W_1=W^{\perp V^\prime}$.
Then, by Lemma~\ref{ort}, $\eps(U_1,W_1)=\eps(W,U)$. Given
$\delta>0$, there exist $x\in U_1$ such that $\|x\|=1$ and
$\|P_{W_1}(x)\|^2\ge (\eps(U_1,W_1))^2-\delta$.

Define the operator $\tilde \Delta : V^\prime\to V^\prime$ by
$\tilde \Delta=P_{V^\prime}\Delta$. Then
$\tilde\Delta=P_{V^\prime}d^* d
P_{V^\prime}=(dP_{V^\prime})^*(dP_{V^\prime})$, whence $\Ker
\tilde \Delta=\Ker \Delta=U$. Therefore,
$\lambda_1(\tilde\Delta)\ge \lambda_1(\Delta)$; this follows
from \eqref{def:lam1} and the fact that
%since for any
%non-negative self-adjoint operator $A:Z\to Z$ we have
%$$\lam_1(A)=\inf\limits_{0\ne v\in (\Ker A)^{\perp Z}}\frac{\la Av,v\ra}{\|v\|^2},$$
%and
$\la \tilde \Delta v,v\ra=\la \Delta v,v\ra$ for any $v\in
V'$. Furthermore, $\Im \tilde\Delta= (\Ker\tilde\Delta)^{\perp
V'}=U_1$, so there exists $g\in V^\prime$ such that $x=\tilde
\Delta g$.

We shall now estimate $\|P_{W_1}(x)\|=\|P_{W_1} (\tilde\Delta
g)\|$ from above. First we have
\begin{multline}
\label{GGmultline1} \|P_{W_1} (\tilde \Delta g)\|^2=\|P_{W_1}
(\Delta g)\|^2\le
\|P_{W^\perp} (\Delta g)\|^2=\\
\sum_{y\in \Vert(Y)}
\|P_{(V^{G_y})^\perp}(\sum_{y=e^+}(g(y)-g(e^-)))\|^2 \leq \rho k
\sum_{y\in \Vert(Y)}
\sum_{y=e^+}\|P_{(V^{G_y})^\perp}(g(y)-g(e^-))\|^2
\end{multline}
where the last inequality holds by Lemma~\ref{Vprime}(b),
definition of $\rho$ and the fact that $(V^{G_y})^{\perp}\in
\Rep_0(G_y)$. We have a similar estimate for $\|P_{W} (\tilde
\Delta g)\|^2$, but without the coefficient $\rho$:
\begin{equation}
\label{GGmultline2} \|P_{W} (\tilde \Delta g)\|^2 \leq\,
k\sum_{y\in \Vert(Y)} \sum_{y=e^+}\|P_{(V^{G_y})}(g(y)-g(e^-))\|^2
\end{equation}
Multiplying \eqref{GGmultline2} by $\rho$ and adding it to
\eqref{GGmultline1}, and using the fact that $\|P_{W_1} (\tilde
\Delta g)\|^2+\|P_{W} (\tilde \Delta g)\|^2=\|\tilde \Delta
g\|^2=1$, we get
\begin{multline*}
\rho  + (1-\rho)\|P_{W_1} (\tilde \Delta g)\|^2\leq \rho
k\sum_{y\in \Vert(Y)} \sum_{y=e^+} \| g(y)-g(e^-)\|^2=\\ \rho
k\sum_{e\in \Edg(Y)} \| (dg)(e)\|^2=
 2\rho k\|dg\|^2
\end{multline*}
(recall that $\|dg\|^2$ is computed with respect to the scalar
product given by \eqref{omega1} with $c(e)=1/2$). Finally, note
that
\begin{equation}
\label{lambdatilde} \|dg\|^2=\la\Delta g,g\ra= \la\tilde\Delta
g,g\ra =\frac{\la\tilde\Delta g,g\ra} {\|\tilde \Delta g\|^2}\leq
\frac{1}{\lam_1(\tilde\Delta)}\leq \frac{1}{\lam_1(\Delta)}.
\end{equation}
Thus $\rho  + (1-\rho)\|P_{W_1} (\tilde \Delta g)\|^2\leq
\frac{2\rho k}{\lam_1(\Delta)}$, and therefore
$$\rho(\{V^{G_y}\})-\delta=(\eps(U_1,W_1))^2-\delta  \le \|P_{W_1} (\tilde \Delta g)\|^2\le
\frac{\rho}{1-\rho}\left( \frac{2k}{\lambda_1(\Delta)}-1\right).
$$
\end{proof}
\subsection{Magic graph on six vertices}

\begin{Definition}\rm Let $G$ be a group  generated by a collection of 6 subgroups
$\{X_{ij}\mid 1\leq i,j\leq 3,\,\, i\neq j\}$ such that for any
permutation $i,j,k$ of the set $\{1,2,3\}$ the following
conditions hold:
\begin{itemize}
\item[(a)] $X_{ij}$ is abelian; \item[(b)] $X_{ij}$ and $X_{ik}$
commute; \item[(c)] $X_{ji}$ and $X_{ki}$ commute; \item[(d)]
$[X_{ij}, X_{jk}]= X_{ik}$.
\end{itemize} Then we will say
that $(G,\{X_{ij}\})$ is an {\it $A_2$-system.} The group $G$
itself will be called an {\it $A_2$-group.}
\end{Definition}
If $G=EL_3(R)$ for some ring $R$ with $1$ and $\{X_{ij}\}$ are root
subgroups, then $(G,\{X_{ij}\})$ is clearly an $A_2$-system. In
the next section we will see that in fact $EL_n(R)$ is an
$A_2$-group for any $n\geq 3$.

Let $Y$ be the graph with 6 vertices $\{(i,j) : 1\le i\ne j\le
3\}$, such that $(i,j)$ is connected to $(k,l)$ if and only if
$\{i,j,k,l\}=\{1,2,3\}$. Each $A_2$-system  $(G,\{X_{ij}\})$ has a
natural decomposition over $Y$:

If $\{i,j,k\}=\{1,2,3\}$, we define the vertex group
$G_{(i,j)}=\la X_{ik},X_{kj}\ra$. Henceforth we will write
$G_{ij}$ for $G_{(i,j)}$. Note that $G_{ij}$ is a nilpotent group
of class two and $[G_{ij},G_{ij}]=X_{ij}$. The edge groups are
defined as follows. If $e\in \Edg(Y)$ connects $(i,j)$ and $(i,k)$,
we set $G_{e}=X_{ij}X_{ik}$. If $e\in \Edg(Y)$ connects $(j,i)$ and
$(k,i)$, we set $G_{e}=X_{ji}X_{ki}$. Finally, if $e\in \Edg(Y)$
connects $(i,j)$ and $(j,k)$, we set $G_e=X_{ik}$.

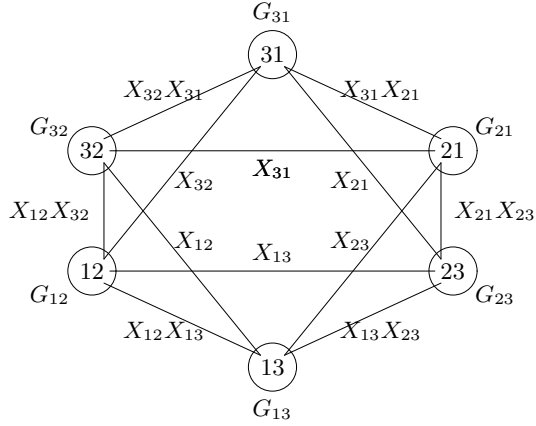
\begin{picture}(200,170)(50,10)
\setlength{\unitlength}{0.08cm} \drawcircle{70.0}{32.0}{8.0}{}
\drawcircle{100.0}{16.0}{8.0}{}
\drawcircle{130.0}{32.0}{8.0}{}\drawcircle{70.0}{52.0}{8.0}{}\drawcircle{100.0}{68.0}{8.0}{}\drawcircle{130.0}{52.0}{8.0}{}

\drawpath{72.0}{30.0}{98.0}{18.0}\drawpath{102.0}{18.0}{128.0}{30.0}\drawpath{128.0}{34.0}{128.0}{50.0}\drawpath{128.0}{54.0}{102.0}{66.0}
\drawpath{98.0}{66.0}{72.0}{54.0}\drawpath{72.0}{50.0}{72.0}{34.0}
\drawpath{73.0}{32.0}{127.0}{32.0}\drawpath{73.0}{52.0}{127.0}{52.0}\drawpath{72.0}{34.0}{98.0}{66.0}\drawpath{102.0}{66.0}{128.0}{34.0}
\drawpath{98.0}{18.0}{72.0}{50.0}\drawpath{102.0}{18.0}{128.0}{50.0}
\drawcenteredtext{70.0}{32.0}{\footnotesize$12$}
\drawcenteredtext{70.0}{52.0}{\footnotesize$32$}
\drawcenteredtext{100.0}{16.0}{\footnotesize$13$}
\drawcenteredtext{100.0}{68.0}{\footnotesize$31$}
\drawcenteredtext{130.0}{32.0}{\footnotesize$23$}
\drawcenteredtext{130.0}{52.0}{\footnotesize$21$}
\drawcenteredtext{63.0}{28.0}{\footnotesize $G_{12}$}
\drawcenteredtext{63.0}{56.0}{\footnotesize $G_{32}$}
\drawcenteredtext{137.0}{28.0}{\footnotesize $G_{23}$}
\drawcenteredtext{137.0}{56.0}{\footnotesize $G_{21}$}
\drawcenteredtext{100.0}{9.0}{\footnotesize $G_{13}$}
\drawcenteredtext{100.0}{75.0}{\footnotesize $G_{31}$}
\drawcenteredtext{63.0}{42.0}{\footnotesize $X_{12}X_{32}$}
\drawcenteredtext{137.0}{42.0}{\footnotesize $X_{21}X_{23}$}
\drawcenteredtext{82.0}{22.0}{\footnotesize $X_{12}X_{13}$}
\drawcenteredtext{118.0}{22.0}{\footnotesize $X_{13}X_{23}$}
\drawcenteredtext{82.0}{62.0}{\footnotesize $X_{32}X_{31}$}
\drawcenteredtext{118.0}{62.0}{\footnotesize $X_{31}X_{21}$}
\drawcenteredtext{100.0}{49.0}{\footnotesize $X_{31}$}
\drawcenteredtext{100.0}{49.0}{\footnotesize $X_{31}$}
\drawcenteredtext{100.0}{35.0}{\footnotesize $X_{13}$}
\drawcenteredtext{87.0}{37.0}{\footnotesize $X_{12}$}
\drawcenteredtext{113.0}{37.0}{\footnotesize $X_{23}$}
\drawcenteredtext{87.0}{47.0}{\footnotesize $X_{32}$}
\drawcenteredtext{113.0}{47.0}{\footnotesize $X_{21}$}
\end{picture}

In this subsection we prove the following theorem, which will be
the main step in the proof of Theorem~\ref{thm:main}.
\begin{Theorem}
\label{6points} Let $(G,\{X_{ij}\})$ be an $A_2$-system, and let
$G_{ij}$ be defined as above. Then $\kappa(G,\bigcup G_{ij})\geq
\frac{3}{8}$ and $\kappa(G,\bigcup X_{ij})\geq \frac{1}{8}$.
\end{Theorem}
\begin{proof}
We begin by computing orthogonality constants between edge groups:
\begin{Claim}\label{rho}
Let $1\leq i\neq j\leq 3$, and let $V\in \Rep_0(G_{ij})$. Then
\begin{itemize}
\item [(a)] $\rho(\{V^{G_e} : e^+=(i,j)\})\le\frac 12$, \item
[(b)] If $V^{X_{ij}}=0$, then $\rho(\{V^{G_e} : e^+=(i,j)\})\le
\frac {1+\sqrt 2}{4\sqrt 2}$.
\end{itemize}
\end{Claim}
\begin{proof} Without loss of generality, we may assume that
$i=1, j=3$. For any $v\in V$ we put $v_l=P_{V^{X_{13}}}(v)$ and
$v_n=P_{(V^{X_{13}})^{\perp}}(v)$. Note that $V^{X_{13}}$ and
$(V^{X_{13}})^{\perp}$ are $G_{13}$-submodules since $X_{13}$ is
normal in $G_{13}$. Therefore,
\begin{equation}
\label{lncomponents} \text{if $v\in V^H$ for some subgroup $H$, we
also have $v_l, v_n\in V^H$}
\end{equation}
Let $e_1,e_2,e_3,e_4\in \Edg(Y)$ be the four edges with
$e_i^+=(1,3)$, and let $H_i=G_{e_i}$ for $1\leq i\leq 4$. For a
suitable ordering of edges we have $H_1= X_{12} X_{13}$, $H_2 =
X_{23} X_{13}$, $H_3= X_{12}$, $H_4 = X_{23}$.

Take any $a\in V^{H_1}$, $b\in V^{H_2}$, $c\in V^{H_3}$, $d\in
V^{H_4}$. Clearly, $a=a_l$ and $b=b_l$. Therefore,
$\|a+b+c+d\|^2=\|(a+b+c+d)_l\|^2+\|(c+d)_n\|^2$. By
\eqref{lncomponents} we have $c_l\in V^{X_{12}}\cap V^{X_{13}}= V^{X_{12} X_{13}}$, and
similarly $d_l\in V^{X_{23}X_{13}}$. By Lemma~\ref{orth_normal},
the subspaces $V^{X_{12}X_{13}}$ and $V^{X_{23}X_{13}}$ are
orthogonal, and thus $(a+c)_l$ is orthogonal to $(b+d)_l$.
Therefore,
\begin{multline*}
\|a+b+c+d\|^2=\|(a+c)_l\|^2+\|(b+d)_l\|^2+\|(c+d)_n\|^2\leq \\
2(\|a_l\|^2+\|c_l\|^2+\|b_l\|^2+\|d_l\|^2+\|c_n\|^2+\|d_n\|^2)=\\
2(\|a\|^2+\|b\|^2+\|c\|^2+\|d\|^2)=4\cdot
\frac{1}{2}(\|a\|^2+\|b\|^2+\|c\|^2+\|d\|^2).
\end{multline*}
Thus, we proved (a).

Now assume that $V^{X_{13}}=\{0\}$. Then $a=b=0$.
%By \eqref{lncomponents}, we have $c_n\in V^{X_{12}}$ and $d_n\in V^{X_{23}}$.
Since $V\in\Rep_0(G_{13})$,
Proposition~\ref{class2:general} yields $\la c, d\ra\leq
\frac{1}{\sqrt{2}}\|c\|\|d\|$. Therefore,
\begin{multline*}
\|a+b+c+d\|^2=\|c+d\|^2\leq (1+\frac{\sqrt{2}}{2})(\|c\|^2+\|d\|^2)=\\
4\cdot \frac {1+\sqrt 2}{4\sqrt 2}(\|a\|^2+\|b\|^2+\|c\|^2+\|d\|^2),
\end{multline*}
which proves (b).
\end{proof}

We proceed with the proof of Theorem~\ref{6points}. Let
$V\in\Rep_0(G)$. We consider the standard Laplace operator
$\Delta=\Delta(Y)$:
$$(\Delta f)(y)=4f(y)-\sum_{y=e^+}f(e^-).$$
It is easy to see that $\lambda_1(\Delta)=4$. Since $Y$ is
$4$-regular, the quantity $\frac{\lam_1(\Delta)}{2k}$ in the
statement of Theorem~\ref{GG1} is equal to $\frac{4}{2\cdot
4}=\frac{1}{2}$. Thus, Theorem~\ref{GG1} would have been
applicable to $G$ if we knew that for each $y\in \Vert(Y)$, the
quantity $\rho(y)=\rho(\{G_e : y=e^+\},G_y)$ was less than $1/2$.
However, Claim~\ref{rho} only shows that $\rho(y)\leq 1/2$. Thus,
we cannot apply Theorem \ref{GG1} directly. However using a
similar argument along with some additional analysis we will
obtain the desired result.

Recall the notations from the proof of Theorem \ref{GG1}. We let
$W$ be the subspace of $\Omega^0(Y)$ consisting of functions
$f:\Vert(Y)\to V$ such that $f(y)\in V^{G_y}$ for all $y\in \Vert(Y)$, and
let $U=\Ker \Delta=\Ker d$ be the subspace of constant functions.
We put $V^\prime=U+W$, $U_1=U^{\perp V^\prime}$ and $W_1=W^{\perp
V^\prime}$. Let
$$\gamma = \rho(\{V^{G_y} : y\in \Vert(Y)\})=(\eps(W,U))^2=(\eps(W_1,U_1))^2.$$
Given $\delta>0$, let $x\in U_1$ be such that $\|x\|=1$ and
$\|P_{W_1}(x)\|^2\ge \gamma-\delta$. Define $\tilde \Delta:V'\to
V'$ by $\tilde \Delta=P_{V^\prime}\Delta$, and let $g\in V^\prime$
be such that $x=\tilde \Delta g$. \vskip .12cm For any function
$h\in \Omega^1(Y)$, define $h_1,h_2,h_3\in \Omega^1(Y)$ by
$$h_1(e)=P_{(V^{X_{e^+}})^\perp}
 (h(e)),\
 h_2(e)=P_{(V^{X_{e^+}})\cap (V^{G_{e^+}})^\perp }(h(e))\
 \textrm{and}\
 h_3(e)=P_{V^{G_{e^+}}}
 (h(e)).$$
Then $h=h_1+h_2+h_3$, and $h_1,h_2,h_3$ are pairwise orthogonal.

The following technical claim will be proved at the end of this
subsection.

\begin{Claim}\label{bounds} The function $dg\in \Omega^1(Y)$ satisfies the following inequalities:
\begin{itemize}
\item[(a)] $\|dg\|^2\le 3\|(dg)_1\|^2+5\|(dg)_3\|^2$, \item[(b)]
$\|dg\|^2\leq\frac{1}{4}$ \item[(c)] $\|(dg)_3\|^2\ge \frac
{1-\|P_{W_1} (\tilde \Delta g)\|^2}8.$
\end{itemize}
\end{Claim}
Using Claims \ref{rho} and \ref{bounds} and Lemma~\ref{Vprime}, we
can estimate $\|P_{W_1} (\tilde \Delta g)\|^2$:
\begin{multline*}  \|P_{W_1} (\tilde \Delta g)\|^2\leq
\sum_{y\in \Vert(Y)}\| P_{(V^{G_y})^\perp}(\sum_{y=e^+}(g(y)-g(e^-))\|^2 \\
= \sum_{y\in \Vert(Y)} \|
P_{(V^{X_y})^\perp}\Big(\sum_{y=e^+}dg(e)\Big)\|^2+
\| P_{(V^{G_y})^\perp\cap V^{X_y}}\Big(\sum_{y=e^+}dg(e)\Big)\|^2\\
\le\ \sum_{y\in \Vert(Y)}4\cdot\frac{1+\sqrt 2}{4\sqrt 2}
\sum_{y=e^+}\|P_{(V^{X_y})^\perp}(dg(e))\|^2+ \sum_{y\in \Vert(Y)}
4\cdot\frac{1}{2} \sum_{y=e^+}\|P_{(V^{G_y})^\perp\cap V^{X_y}}
(dg(e))\|^2\\= (2+\sqrt
2)\|(dg)_1\|^2+4\|(dg)_2\|^2=4\|dg\|^2-(2-\sqrt
2)\|(dg)_1\|^2-4\|(dg)_3\|^2\\ \le
(4-\frac{2-\sqrt 2}3)\|dg\|^2-(4-\frac{5(2-\sqrt 2)}3)\|(dg)_3\|^2\\
\leq\frac{10+\sqrt 2}{12}-\frac{(2+5\sqrt 2)(1-\|P_{W_1} (\tilde
\Delta g)\|^2)}{24}.
\end{multline*}
From the above inequality it follows that $\|P_{W_1} (\tilde
\Delta g)\|^2)\le \frac{18-3\sqrt 2}{22-5\sqrt 2}.$ Thus, by the
choice of $g$ we have
$$1-\rho(\{V^{G_y}\})=1-\gamma\geq 1-\|P_{W_1} (\tilde \Delta g)\|^2-\delta=
\frac{4-2\sqrt{2}}{22-5\sqrt
2}-\delta=\frac{2}{17+6\sqrt{2}}-\delta.$$ Since $\delta$ is
arbitrary, by Lemma~\ref{orthogT} we get $\kappa(G,\bigcup
G_{ij})\geq \frac{\sqrt{2}\cdot\sqrt{2}}{\sqrt{17+6\sqrt
2}}\geq\frac{3}{8}$.

Finally, since $G_{ij}=X_{ik}X_{kj}X_{ij}$, the Kazhdan ratio
$\kappa_r(G,\bigcup G_{ij};\bigcup X_{ij})$ is at least $1/3$,
whence $\kappa(G,\bigcup X_{ij})\geq \frac{3}{8}\cdot
\frac{1}{3}=\frac{1}{8}$.
\end{proof}
\begin{proof}[Proof of Claim~\ref{bounds}]
(b) is proved by the same argument as in \eqref{lambdatilde}, and
(c) easily follows from \eqref{GGmultline2}, so we only need to
establish (a). For brevity, in the following computation we will
write $g_{ij}$ for $g((i,j))$.

Let $\{i,j,k\}=\{1,2,3\}$. First we claim that
$$P_{V^{X_{ij}}}(g_{ik}-g_{jk})=P_{V^{G_{ik}}}(g_{ik}-g_{jk}).$$
Indeed, let $z=g_{ik}-g_{jk}$. Then $z\in V^{X_{ik}X_{jk}}$ by
Lemma~\ref{Vprime}(a). Since $X_{ij}$ is normalized by $X_{ik}$
and $z$ is $X_{ik}$-invariant, we conclude that
$P_{V^{X_{ij}}}(z)$ is also $X_{ik}$-invariant, so
$P_{V^{X_{ij}}}(z)=P_{V^{X_{ij}X_{ik}}}(z)$. Similarly,
$X_{ij}X_{ik}$ is normalized by $X_{jk}$, and thus
$P_{V^{X_{ij}X_{ik}}}(z)=P_{V^{G_{ik}}}(z)$ since
$G_{ik}=X_{ij}X_{ik}X_{jk}$.

Therefore, \begin{multline}
\label{est1}
\|g_{ik}-g_{jk}\|^2=\|P_{V^{X_{ij}}}(g_{ik}-g_{jk})\|^2
+\|P_{(V^{X_{ij})^\perp}}(g_{ik}-g_{jk})\|^2\\
=\|P_{V^{G_{ik}}}(g_{ik}-g_{jk})\|^2 +\|P_{(V^{X_{ij})^\perp}}((g_{ik}-g_{ij})+(g_{ij}-g_{jk}))\|^2\\
=\|P_{V^{G_{ik}}}(g_{ik}-g_{jk})\|^2+\|P_{(V^{X_{ij})^\perp}}(g_{ij}-g_{jk})\|^2.\end{multline}
where the last equality holds since $g_{ik}-g_{ij}\in V^{X_{ij}}$.

Using a similar argument we get
\begin{equation}
\label{est2} \|g_{ik}-g_{ij}\|^2\leq
\|P_{V^{G_{ik}}}(g_{ik}-g_{ij})\|^2+\|P_{(V^{X_{jk})^\perp}}(g_{jk}-g_{ij})\|^2
\end{equation}

Next we estimate $\|g_{ij}-g_{jk}\|$. Note that
$P_{(V^{G_{ik}})^\perp} (g_{ij}-g_{ik})$ is orthogonal to
$P_{(V^{G_{ik}})^\perp} (g_{ik}-g_{jk})$ by
Lemma~\ref{orth_normal} since $g_{ij}-g_{ik}\in V^{X_{ij}X_{ik}}$
and $g_{ik}-g_{jk}\in V^{X_{ik}X_{jk}}$. Therefore,
\begin{multline}
\label{est3}
\|g_{ij}-g_{jk}\|^2= \|P_{V^{G_{ik}}}(g_{ij}-g_{jk})\|^2+\\
\|P_{(V^{G_{ik}})^\perp}(g_{ij}-g_{ik})\|^2+\|P_{(V^{G_{ik}})^\perp}(g_{ik}-g_{jk})\|^2\leq
2(\|P_{V^{G_{ik}}}(g_{ij}-g_{ik})\|^2+\|P_{V^{G_{ik}}}(g_{ik}-g_{jk})\|^2)+\\
(\|g_{ij}-g_{ik}\|^2-\|P_{V^{G_{ik}}}(g_{ij}-g_{ik})\|^2)+
(\|g_{ik}-g_{jk}\|^2-\|P_{V^{G_{ik}}}(g_{ik}-g_{jk})\|^2)\leq \\
2\left(\|P_{V^{G_{ik}}}(g_{ij}-g_{ik})\|^2+\|P_{V^{G_{ik}}}(g_{ik}-g_{jk})\|^2\right)+
\|P_{(V^{X_{ij})^\perp}}(g_{ij}-g_{jk})\|^2+
\|P_{(V^{X_{jk})^\perp}}(g_{jk}-g_{ij})\|^2,
\end{multline}
where the last inequality holds by \eqref{est1} and \eqref{est2}.

Finally, combining \eqref{est1}, \eqref{est2} and \eqref{est3}, we
get
\begin{multline*}
2\|dg\|^2=\sum\limits_{\{i,j,k\}=\{1,2,3\}}
\|g_{ik}-g_{jk}\|^2+\|g_{ik}-g_{ij}\|^2+2\|g_{ij}-g_{jk}\|^2\leq \\
\sum\limits_{\{i,j,k\}=\{1,2,3\}}
5\left(\|P_{V^{G_{ik}}}(g_{ik}-g_{ij})\|^2+\|P_{V^{G_{ik}}}(g_{ik}-g_{jk})\|^2\right)+\\
\sum\limits_{\{i,j,k\}=\{1,2,3\}}
3\left(\|P_{(V^{X_{ij})^\perp}}(g_{ij}-g_{jk})\|^2+ \|P_{(V^{X_{jk})^\perp}}(g_{jk}-g_{ij})\|^2\right)=\\
10\|(dg)_3\|^2+6\|(dg)_1\|^2.
\end{multline*}
\end{proof}

\subsection{Spectral criterion. Weighted version}
In this subsection we present the weighted version of our spectral criterion.
In order to formulate this version we need
to generalize the notion of codistance introduced in Subsection~2.2.

\begin{Definition}\rm Let $V$ be a Hilbert space, $X$ a finite set,
$\{U_x\}_{x\in X}$ subspaces of $V$ and $\alpha:X\to\R_{>0}$ a
function. Consider the Hilbert space $\Omega_\alpha(X,V)=\{f:X \to
V\}$ with the following scalar product
$$\la f,g \ra=\sum_{x\in X}
\frac{\la f(x),g(x)\ra}{\alpha (x)}.$$ Let $U=\{f\in
\Omega_\alpha(X,V) : f(x)\in U_x \mbox{ for each }x\in X \}$ and
let $diag(V)$ be the subspace of constant functions. The quantity
$$\rho_\alpha(\{U_x\})=\eps(U,diag(V))^2$$ will
be called the {\it $\alpha$-codistance } between the subspaces
$\{U_x\}_{x\in X}$. It is easy to see that
$$\rho_\alpha(\{U_x\})=
\frac{\sup\left\{\frac{\|\sum_{x\in X}u_x\|^2}{\sum_{x\in
X}\|u_x\|^2\alpha(x)} : u_x\in U_x\right\}}{\sum_{x\in X}\frac1{
\alpha(x)}}.$$
\end{Definition}
Thus, the codistance $\rho(\{U_x\})$ introduced in Section~2
corresponds to the case $\alpha(x)=1$.
\begin{Definition}\rm Let $G$ be a group, $X$ a finite set,  $\{H_x\}_{x\in X}$ a set of subgroups of $G$ and
$\alpha:X\to\R_{>0}$ a function. The {\it $\alpha$-codistance }
between $\{H_i\}$ in $G$, denoted $\rho_\alpha(\{H_i\},G)$, is
defined to be the supremum of the set
$$\{\rho_\alpha(V^{H_1},\ldots, V^{H_n}): V\in\Rep_0(G)\}.$$
If $G$ is generated by $\{H_i\}$ we simply write
$\rho_\alpha(\{H_i\})$ instead $\rho_\alpha(\{H_i\},G)$.
\end{Definition}
Note that while $\rho_\alpha(\{H_i\})$ depends on $\alpha$, it is easy
to see that $\rho_\alpha(\{H_i\})<1$ if and only if $\rho(\{H_i\})<1$.
\vskip .1cm

If $G$ is generated by two subgroups $H_1$ and $H_2$ we have the
following equality.
\begin{equation}\label{relation}\left(\eps(H_1,H_2)\right)^2=
\left(\frac{(\alpha(1)+\alpha(2))\rho_\alpha(\{H_1,H_2\})}{\alpha(1)}-1\right)
\left(\frac{(\alpha(1)+\alpha(2))
\rho_\alpha(\{H_1,H_2\})}{\alpha(2)}-1\right).\end{equation}

\begin{Theorem}
\label{GG2} Let $Y$ be a finite connected graph, let $G$ be a
group with a chosen decomposition over $Y$, and let $c:\Edg(Y)\to
\R_{>0}$ be a function. For each $y\in \Vert(Y)$, we set
$$\alpha(y)=\frac 1{\rho_{c}(\{G_e : y=e^+\},G_y)\sum_{y=e^+}\frac{1}{c(e)} }. $$   Let
  $\Delta$ be the Laplacian of $Y$ corresponding to the weight
functions $\alpha$ and $c$ and assume that $\lambda_1(\Delta)>1$.
Then
 $\kappa(G,\bigcup\limits_{y\in \Vert(Y)}G_y)>0$.
\end{Theorem}
A few remarks are in order. The functions $\alpha:\Vert(Y)\to\R_{>0}$
(which depends  on $c$) and $c:\Edg(Y)\to \R_{>0}$  can be thought of
as weights on the sets of vertices and edges of $Y$.
%Note that
%the Laplacian $\Delta$ has a simple expression in terms of
%$\alpha$ and $\beta$:
%$$(\Delta f)(y)=\sum_{y=e^+}\frac{\alpha(y)}{\beta(e)+\beta(\bar e)}(f(y)-f(e^-)).$$
The seemingly complicated expression for $\alpha$ is designed to
satisfy the following property for each $y\in \Vert(Y)$:
\begin{equation}\label{beta}
\alpha(y) \|\sum_{y=e^+}v_e\|^2\le \sum_{y=e^+}c(e)\|v_e\|^2
\mbox{ whenever } v_e\in (V^{G_e})\cap (V^{G_y})^\perp.
\end{equation}
The inequality (\ref{beta}) holds by the definition of $\rho_{c}$
and $\alpha$.

\begin{proof}[Proof of Theorem~\ref{GG2}] We will follow the same
scheme as in the proof of Theorem~\ref{GG1}. Let $V\in\Rep_0(G)$.
As before, $W$ denotes the space of functions $f\in \Omega^0(Y)$
such that $f(y)\in V^{G_y}$ for all $y\in \Vert(Y)$ and $U=\Ker
\Delta=\Ker d$ is the subspace of constant functions. Note that
$\Omega^0(Y)$ is defined by \eqref{omega0} with $\alpha$ as in the
statement of Theorem~\ref{GG2}, so $\eps(W,U)^2$ equals
$\rho_\alpha(\{V^{G_y}\})$, but not necessarily
$\rho(\{V^{G_y}\})$.

As in the proof of Theorem~\ref{GG1}, we set $V^\prime=U+W$,
$U_1=U^{\perp V^\prime}$ and $W_1=W^{\perp V^\prime}$, and we have
$\eps(U_1,W_1)=\eps(W,U)$. Given $\delta>0$, let $x\in U_1$ such
that $\|x\|=1$ and $\|P_{W_1}(x)\|^2\ge
\rho_\alpha(\{V^{G_y}\})-\delta$. Define $\tilde \Delta:
V^\prime\to V^\prime$ by $\tilde\Delta=P_{V'}\Delta$, and choose
$g\in V^\prime$ such that $x=\tilde \Delta g$. Lemma~\ref{Vprime}
clearly holds. \vskip .1cm
By the definition of $\Delta$ and scalar product on $\Omega^0(Y)$ we have
\begin{equation*} \|P_{W^\perp} (\Delta g)\|^2\\=
\sum_{y\in \Vert(Y)}
\frac{1}{\alpha(y)} \|P_{(V^{G_y})^{\perp}}(\Delta g (y))\|^2=
\sum_{y\in
\Vert(Y)}\alpha(y)\|\big(\sum_{y=e^+}P_{(V^{G_y})^\perp}(dg(e))\big)\|^2.
\end{equation*}
Now applying (\ref{beta}), we obtain
\begin{multline*}\|P_{W_1}
(\tilde \Delta g)\|^2\le\|P_{W^\perp} (\Delta g)\|^2 \leq
\sum_{y\in \Vert(Y)} \sum_{y=e^+}
c(e)\|P_{(V^{G_y})^\perp}(dg(e))\|^2\\\le \|d g\|^2 \le \frac1{\lambda_1(\Delta)},
\end{multline*}
where the last inequality holds by \eqref{lambdatilde}.
Therefore,
$\rho_\alpha(\{V^{G_y}\})\le \frac{1}{\lambda_1(\Delta)}<1$.
Thus $\rho(\{G_y : y\in \Vert(Y)\})<1$, and so
$\kappa(G,\cup G_y)>0$.
\end{proof}
\begin{Remark} Using an argument similar to the one presented in
the proof of Theorem~\ref{GG1}, it is possible to show that
$$\rho_\alpha(\{G_y : y\in \Vert(Y)\})\le\frac{1}{1-\rho}( \frac{1}{\lambda_1(\Delta)}-\rho),$$
where $\rho=\min
\left\{\frac{c(e)}{\alpha(e^+)\deg(e^+)} : e\in \Edg(Y)\right\}.$
With this remark, Theorem~\ref{GG2} in the case $c(e)=1/2$ and $Y$ regular is equivalent to Theorem~\ref{GG1} (note that in this special case
the Laplacian in the statement of Theorem~\ref{GG2} is a scalar multiple of the standard Laplacian).
\end{Remark}

\subsection{The triangle graph}
In this subsection we use Theorem \ref{GG2} to obtain a slight
improvement of Corollary \ref{T3}.

Let $G$ be a group, let $H_1, H_2, H_3$ be subgroups of $G$ such
that $G=\la H_1, H_2, H_3\ra$, and let $Y$ be the complete graph
with 3 vertices $\{1,2,3\}$. Consider the standard decomposition
of $G$ over $Y$: the vertex groups are $G_1=\la H_2,H_3\ra$,
$G_2=\la H_3,H_1\ra$ and $G_3=\la H_1,H_2\ra$ and edge groups are
$G_{(1,2)}=G_{(2,1)}=H_3$, $G_{(2,3)}=G_{(3,2)}=H_1$ and
$G_{(3,1)}=G_{(1,3)}=H_2$.

\begin{Theorem}
\label{T3graph}  Assume that $H_1$ and $H_2$ are
$\eps_3$-orthogonal, $H_2$ and $H_3$ are $\eps_1$-orthogonal,
$H_3$ and $H_1$ are $\eps_2$-orthogonal for some
$\eps_1,\eps_2,\eps_3$ such that
$$\eps_1^2+\eps_2^2+\eps_3^2+2\eps_1\eps_2\eps_3<1,$$
Then $\rho(G_1,G_2,G_3)<1$, and therefore $\kappa(G,H_1\cup
H_2\cup H_3)> 0$.
\end{Theorem}
\begin{Remark} Note that if $\eps_1=\eps_2=\eps$ for some $\eps$, the above
inequality on $\eps_1,\eps_2,\eps_3$ holds if and only if
$\frac{\sqrt{2}\,\eps}{\sqrt{1-\eps_3}}<1$. Thus, the criterion
for positivity of $\kappa(G,H_1\cup H_2\cup H_3)$ in
Corollary~\ref{T3} is a special case of Theorem~\ref{T3graph}.
\end{Remark}
\vskip .12cm

First we prove an auxiliary result.
\begin{Lemma} \label{system} Let $\eps_1,\eps_2,\eps_3$ be
non-negative numbers such that
$\eps_1^2+\eps_2^2+\eps_3^2+2\eps_1\eps_2\eps_3<1$. Then the
system of equations
$$\left \{ \begin{array}{l}
x(u-z)=\eps_1^2\\
y(u-x)=\eps_2^2\\
z(u-y)=\eps_3^2\\
u(u^2-\eps_1^2-\eps_2^2-\eps_3^2)=2\eps_1\eps_2\eps_3\end{array}\right
.$$ has a solution $(x_0,y_0,z_0,u_0)$ satisfying $x_0,y_0,z_0\ge
0$ and $\sqrt{\eps_1^2+\eps_2^2+\eps_3^2}\le u_0<1$.
\end{Lemma}
\begin{proof}
Consider
$f(u)=u(u^2-\eps_1^2-\eps_2^2-\eps_3^2)-2\eps_1\eps_2\eps_3$. Then
$f$ has absolute minimum on $[0,\infty)$ at
$u=\sqrt{\frac{\eps_1^2+\eps_2^2+\eps_3^2}{3}}$ and
$f(\sqrt{\eps_1^2+\eps_2^2+\eps_3^2})<0$. By the hypothesis of the
theorem, $f(1)>0$. Thus, there exists unique $u_0$ such that
$\sqrt{\eps_1^2+\eps_2^2+\eps_3^2}\le u_0<1$ and $f(u_0)=0$.

Substituting $y=\frac{\eps_2^2}{u_0-x}$ and
$x=\frac{\eps_1^2}{u_0-z}$ in the third equation of the system, we
obtain the following equation on $z$:
\begin{equation}
\label{cuadr}(u_0^2-\eps_2^2)z^2-(u_0^3-\eps_1^2u_0-\eps_2^2u_0+\eps_3^2u_0)z+\eps_3^2(u_0^2-\eps_1^2)=0.
\end{equation}
Its discriminant is equal to
$$u_0^2(u_0^2-\eps_1^2-\eps_2^2-\eps_3^2)^2-4\eps_1^2\eps_2^2\eps_3^2=0.$$
Thus, $z_0=\eps_3\sqrt{\frac{u_0^2-\eps_1^2}{u_0^2-\eps_2^2}}$ is
a solution of \eqref{cuadr}, and if we set
$x_0=\frac{\eps_1^2}{u_0-z_0}$ and $y_0=\frac{\eps_2^2}{u_0-x_0}$,
the quadruple $(x_0,y_0,z_0,u_0)$ is a solution to the system. It
is clear from the formula for $z_0$ that $\frac{\eps_3^2}{u_0}\leq
z_0\leq u_0$, so from the first and third equations of the system
we obtain that $x_0\geq 0$ and $y_0\geq 0$.
\end{proof}
\begin{proof}[Proof of Theorem \ref{T3graph}]
Without loss of generality we may assume that
$\epsilon_i=\epsilon(H_j,H_k)$ if $\{i,j,k\}=\{1,2,3\}$.

Let $u_0,x_0,y_0,z_0$ satisfy the conclusion of
Lemma~\ref{system}. We apply Theorem \ref{GG2} with  $c$ defined
by the following table \vskip .2cm

\begin{tabular}{|c|c|c|c|}
  \hline
  $e^-/e^+$ & 1 & 2 & 3 \\
  \hline
  1 & * & $1+u_0-x_0$ & $1+z_0$\\
  2 & $1+x_0$ & * & $1+u_0-y_0 $\\
  3 & $1+u_0-z_0$& $1+y_0$ & * \\
  \hline
\end{tabular}
\vskip .2cm

Let $\alpha$ be the function from Theorem~\ref{GG2}. We claim that $\alpha(p)=1$ for every $p\in \Vert(Y)$.
We put $\rho(p)=(\sum_{p=e^+}c(e)^{-1})^{-1}$ for all $p\in \Vert(Y)$; we then need to show
that $\rho(p)=\rho_{c}(\{G_e : p=e^+\},G_p)$. For instance, consider $p=1$. Then
$$\rho(1)=\frac{1}{c((2,1))^{-1}+c((3,1))^{-1}}=\frac{c(2,1)c(3,1)}{c(2,1)+c(3,1)}=
\frac{(1+x_0)(1+u_0-z_0)}{1+x_0+1+u_0-z_0}$$ and, by our choice of
$u_0,x_0,y_0,z_0$, we obtain that
$$\eps_1^2=\left(\frac{(1+x_0+1+u_0-z_0)\rho(1)}{1+x_0}-1\right)\left(\frac{(1+x_0+1+u_0-z_0)\rho(1)}{1+u_0-z_0}-1\right).$$
Thus, $\rho(1)= \rho_{(1+x_0,1+u_0-z_0)}(H_2,H_3)$ by
(\ref{relation}) because $\eps_1= \eps(H_2,H_3)$. For the vertices
$2$ and $3$ the argument is similar.

Therefore, for any $e\in \Edg(Y)$ we have
$$\frac{\alpha(e^+)}{c(e)+c(\bar e)}= \frac {1}{2+u_0}.$$ Hence
$$(\Delta f)(p)=\frac 1{2+u_0}(2f(p)-\sum_{p=e^+}f(e^-)),$$  and so
$\lambda_1(\Delta)=\frac 3{2+u_0}>1$. Thus, the result follows
from Theorem \ref{GG2}.
\end{proof}
\vskip .1cm
{\bf Addendum:} After an earlier version of this paper was distributed,
Kassabov used a generalization of the techniques presented in Sections~2 and 3
of this paper to prove the following striking result:

\begin{Theorem}[Kassabov]
\label{posdef} Let $G$ be a group generated by subgroups
$H_1,\ldots, H_n$ (where $n\geq 2$), and let $\eps_{ij}=\eps(H_i,
H_j)$  for $i\neq j$. Let $E=(e_{ij})$ be the $n\times n$ matrix
defined by $e_{ii}=1$ and $e_{ij}=-\eps_{ij}$ for $i\neq j$, and
assume that $E$ is positive definite. Then $\kappa(G,\cup H_i)>0$.
%In particular,
%if each pair $(G,H_i)$ has relative property $(T)$, then
%$G$ has relative property $(T)$.
\end{Theorem}

It is easy to see that the matrix $E$ is positive definite in the following two
special cases:
\begin{itemize}
\item[(i)] $\max\{\eps_{ij} : i\neq j\} < \frac{1}{n-1}$

\item[(ii)] $n=3$ and $\eps_{12}^2+\eps_{23}^2+\eps_{13}^2+2\eps_{12}\eps_{23}\eps_{13}<1$.
\end{itemize}
Thus, Theorem~\ref{posdef} generalizes both Corollary~\ref{Tn} and Theorem~\ref{T3graph}.

\section{Property $(T)$ for $EL_n(R)$}
In this section we present the main applications of our method. In
the first subsection we use Theorem~\ref{6points} to prove that if
$R$ is a finitely generated ring with 1, then $EL_n(R)$ has
property $(T)$ for all $n\ge 3$ (Theorem~\ref{thm:main}). We shall
also establish the analogous result for Steinberg groups. In the
second subsection we give an alternative proof of property $(T)$
under some additional assumptions on $n$ or $R$. This proof uses
only Corollary~\ref{Tn} and results of Section~4 and naturally
yields a finitely presented cover of $EL_n(R)$ with $(T)$. In the
last subsection we discuss possible generalizations of
Theorem~\ref{thm:main} and describe a counterexample to a
conjecture of Lubotzky and Weiss.

Throughout the section we fix an integer $n\geq 3$ and a
finitely generated associative ring $R$ with $1$. For $i,j\in \{1,\ldots,n\}$,
with $i\neq j$, and $r\in R$ let $e_{ij}(r)\in EL_n(R)$ denote the elementary
matrix whose $(i,j)$-entry is equal to $r$ and all other
non-diagonal entries are equal to $0$.

If $l,m\in\dbZ$, with $l\leq m$, by $[l,m]$ we denote the set
$\{i\in \dbZ\mid l\leq i\leq m\}$.

Let $a=[\frac{n}{3}]$, $b=[\frac{(n+1)}{3}]$ and $c=[\frac{(n+2)}{3}]$
(where $[x]$ is the integer part of $x$),
so that $a+b+c=n$, and let $\calI_1=[1,a]$, $\calI_2=[a+1,a+b]$, $\calI_3=[a+b+1,a+b+c]$.
\vskip .1cm

\subsection{Proof of Theorem \ref{thm:main}}
\label{ELnR1} We fix a generating set $\{x_0,x_1,\ldots, x_d\}$
for $R$, where $x_0=1$. It is clear that the set
\begin{equation}
\label{eq:genset} \Sigma=\{e_{ij}(x_m)\, : \,i,j\in
\{1,\ldots,n\}, i\neq j,\,\, 0\leq m\leq d\}
\end{equation}
generates $EL_n(R)$. For each $i\ne j\in \{1,2,3\}$, we define the
subgroup $X_{ij}$ of $EL_n(R)$ by
$$X_{ij}=\la e_{kl}(r)\ :\ k\in \calI_i,\,\, l\in \calI_j,\,\, r\in R\ra.$$
In other words,
\begin{align*}
\small
&X_{12}= \left \{\left (\begin{array}{ccc} I_a & * & 0\\
0 & I_b & 0\\
0 & 0 & I_c \end{array}\right )\right\}
&
&X_{23}= \left \{\left (\begin{array}{ccc} I_a & 0 & 0\\
0 & I_b & *\\
0 & 0 & I_c \end{array}\right )\right\}
&
&&X_{31}= \left \{\left (\begin{array}{ccc} I_a & 0 & 0\\
0 & I_b & 0\\
* & 0 & I_c \end{array}\right )\right\}
&\\
&X_{21}= \left \{\left (\begin{array}{ccc} I_a & 0 & 0\\
* & I_b & 0\\
0 & 0 & I_c \end{array}\right )\right\}
&
&X_{32}= \left \{\left (\begin{array}{ccc} I_a & 0 & 0\\
0 & I_b & 0\\
0 & * & I_c \end{array}\right )\right\}
&
&&X_{13}= \left \{\left (\begin{array}{ccc} I_a & 0 & *\\
0 & I_b & 0\\
0 & 0 & I_c \end{array}\right )\right\}
&
\end{align*}
where $*$ stands for an arbitrary matrix of appropriate size with
entries in $R$. \vskip .2cm It is clear that
$(EL_n(R),\{X_{ij}\})$ is an $A_2$-system\footnote{This
observation may be thought of as a ``generalization'' of a
well-known property that for $n=3k$ the group $EL_n(R)=EL_{3k}(R)$
is naturally isomorphic to $EL_3(M_k(R))$. This isomorphism plays crucial
role in many proofs in \cite{Ka2}.}, and so by Theorem
\ref{6points} we have
$$\kappa(EL_n(R),\cup_{i,j} X_{ij})\geq\frac{1}{8}.$$
In order to finish the proof of property $(T)$ for $EL_n(R)$, we
use the following result which is a special case of \cite[Corollary~1.10]{Ka2}:

\begin{Proposition}[Kassabov]
Let $V$ be a unitary representation of $EL_n(R)$ and let $v\in V$
be a $(\Sigma,\eps)$-invariant vector (for some $\eps>0$). Then
for any $g\in \cup_{i,j} X_{ij}$ we have
$$\|gv-v\| < (12\sqrt {2d}+2\sqrt {3n}+36\sqrt 2)\cdot \eps \| v\|.$$
In other words, $\kappa_r(EL_n(R),\cup_{i,j} X_{ij};\Sigma)\geq \frac{1}{12\sqrt {2d}+2\sqrt{3n}+36\sqrt 2}$
(where as before $\kappa_r$ is the Kazhdan ratio).
\label{Kassabov2}
\end{Proposition}
From Proposition~\ref{Kassabov2} and \eqref{CCC} it follows that
$$
\kappa(EL_n(R),\Sigma)\geq \frac{\kappa(EL_n(R),\cup
_{i,j}X_{ij})}{12\sqrt{2d}+2\sqrt{3n}+36\sqrt 2}>0.
$$
Since $\Sigma$ is finite, we conclude that $EL_n(R)$ has property
$(T)$, and moreover
\begin{equation}
\label{Kazhconst_main} \kappa(EL_n(R),\Sigma)\geq
\frac{1}{8(12\sqrt {2d}+2\sqrt {3n}+36\sqrt 2)}.
\end{equation}

We shall now discuss the analogue of Theorem~\ref{thm:main} for Steinberg groups.

\begin{Definition}\rm
Let $n\geq 3$. The {\it Steinberg group} $St_n(R)$ is the group generated by the
symbols $\{E_{ij}(r): 1\leq i\neq j\leq n,\, r\in R\}$ subject to the following relations:
\begin{itemize}
\item[(St1)] $E_{ij}(r)E_{ij}(s)=E_{ij}(r+s)$
\item[(St2)] $[E_{ij}(r),E_{kl}(s)]=1$  if  $i\neq l, k\neq j$
\item[(St3)] $[E_{ij}(r),E_{jk}(s)]=E_{ik}(rs)$ if $i\neq k$.
\end{itemize}
\end{Definition}

There is a natural surjective homomorphism $\pi_{st}: St_n(R)\to
EL_n(R)$ given by $\pi_{st}(E_{ij}(r))=e_{ij}(r)$. As in the case
of $EL_n(R)$, if $\{x_0=1,x_1,\ldots, x_d\}$ is a
generating set for $R$, then $St_n(R)$ is generated by the set
\begin{equation}
\label{eq:sigmast}
\Sigma^{st}=\{E_{ij}(x_m)\, : \,i,j\in
\{1,\ldots,n\}, i\neq j,\,\, 0\leq m\leq d\}.
\end{equation}
The following is the version of Theorem~\ref{thm:main} for
Steinberg groups, with explicit Kazhdan constant:

\begin{Theorem}
\label{Steinberg} The Steinberg group $St_n(R)$, $n\ge 3$,  has
property $(T)$. Furthermore,
$$\kappa(St_n(R), \Sigma^{st})\geq \frac{1}{8(12\sqrt {2d}+2\sqrt {3n}+36\sqrt 2)}.$$
\end{Theorem}
The proof of Theorem~\ref{Steinberg} is virtually identical to
that of Theorem~\ref{thm:main}, except that Proposition~\ref{Kassabov2}
has to be replaced by the following generalization.

\begin{Proposition}
For $1\leq i,j\leq 3$ let $\Xgal_{ij}=\la E_{kl}(r) : k\in \calI_i,\,\, l\in \calI_j,\,\, r\in R\ra$
be the subgroup of $St_n(R)$ ``corresponding'' to $X_{ij}$. Then
$$\kappa_r(St_n(R),\cup_{i,j} \Xgal_{ij};\Sigma^{st})\geq \frac{1}{12\sqrt {2d}+2\sqrt{3n}+36\sqrt 2}$$

\label{Kassabov_St}
\end{Proposition}
Proposition~\ref{Kassabov_St} cannot be deduced from the results stated in \cite{Ka2};
however, the proof of Proposition~\ref{Kassabov2} in \cite{Ka2} can be applied
to Proposition~\ref{Kassabov_St} almost without changes. For the convenience
of the reader we present this argument in Appendix~A.
\vskip .1cm

By a theorem of Krsti\' c and McCool~\cite[Theorem~3]{KrM}, the Steinberg group $St_n(R)$ is finitely
presented for any $n\geq 4$ and any finitely presented ring $R$, in particular for
$R=\dbZ\la x_1,\ldots, x_d\ra$. Thus, for any $n\geq 4$ and any associative ring $R$ generated by
$d$ elements the group $St_n(\dbZ\la x_1,\ldots, x_d\ra)$ is a finitely
presented cover with $(T)$ for $EL_n(R)$. By \cite[Corollary~2]{KrM}, the group
$St_3(R)$ is not finitely presented whenever $R$ surjects onto $F[t]$ for some field $F$.
However, in the next subsection we will construct a finitely presented cover with $(T)$ for $EL_3(R)$
if $R$ is an algebra over a finite field $F$, with $|F|\geq 5$.

\subsection{A finitely presented cover of $EL_n(R)$ with property $(T)$}
In this subsection we give the second proof of property $(T)$
for $EL_n(R)$ under additional assumptions that $n\geq 7$ or
$R$ is an algebra over a finite field $F$, with $|F|\geq 5$.
The method uses only a finite number of relations of
$EL_n(R)$ and thus provides an (explicit) finitely presented cover
of $EL_n(R)$ with property $(T)$. The Kazhdan constant will be asymptotically
smaller than the one yielded by the first method when both $n$ and $d$ go to infinity
(where $d$ is the number of generators of $R$), but for small $d$,
this method will produce a better constant
(see Propositions~\ref{Kazhdan_ELn} and \ref{Kazhdan_ELn2}).
An explicit finite presentation for a cover of $EL_n(R)$ with property $(T)$
is given at the end of the subsection (see Theorem~\ref{explicitpres}).
In view of the discussion at the end of subsection~6.1, these results
are of most interest in the case when $n=3$ and $R$ is an algebra over a finite field.
However, even in the remaining cases, the presentation given by Theorem~\ref{explicitpres}
contains significantly fewer relations than the finite presentation for $St_n(R)$ constructed in \cite{KrM}.
\vskip .1cm

Recall that we fixed an associative ring $R$. In this subsection we
shall assume that $R$ is an $R_0$-algebra where $R_0$ is either
$\dbZ$ (integers) or a finite field. As before, let $x_0=1$,
and let $\{x_1,\ldots, x_d\}$ be a set which generates $R$
as an $R_0$-algebra. If $R_0=\dbZ$ or $R_0$ is a prime field
(that is $|R_0|$ is prime), then $\{x_0,\ldots, x_d\}$ generates $R$ as a ring (so $x_i$ have
the same meaning as in subsection~6.1), and
if $R_0$ is a non-prime field, $R$ is generated as a ring
by $\{x_1,\ldots, x_d\}$ and one additional element.
We also fix a basis $\{\alpha_1=1,\alpha_2,\ldots,\alpha_s\}$ of $R_0$ over $\dbZ$.

Recall that $a=[\frac{n}{3}]$, $b=[\frac{(n+1)}{3}]$
and $c=[\frac{(n+2)}{3}]$ and
$\calI_1=[1,a]$, $\calI_2=[a+1,a+b]$, $\calI_3=[a+b+1,a+b+c]$.
For each $i,j\in
\{1,2,3\}$ and $0\le m\le d$, we define the following subsets of $EL_n(R)$:
$$\Sigma_{ij}(m)=\{ e_{kl}(\alpha_t x_m)\ :\ k\in \calI_i,\,\, l\in \calI_j,\,\, k\ne l,\,\, 1\le t\le s
\}.$$ We put
$$ \Sigma_{ij}=\bigcup\limits_{m=0}^d \Sigma_{ij}(m) \mbox{ and }\Sigma=\bigcup\limits_{i,j} \Sigma_{ij}.$$
Clearly, $\Sigma$ is a generating set for $EL_n(R)$.
Note that if $R_0$ is $\dbZ$ or a prime field, the definition of $\Sigma$ coincides with \eqref{eq:genset}.
\vskip .15cm
{\bf Construction of a finitely presented cover $\Gamma$.}
Now we shall describe a finite set of relations of $EL_n(R)$
with respect to the generating set $\Sigma$, which are sufficient to define
a group with property $(T)$, provided
$n\geq 7$ or $R_0$ is finite, with $|R_0|\geq 5$.
Without loss of generality we can (and will) assume
that $R$ is the free associative algebra $R_0\la x_1,\ldots, x_d\ra$
(since $EL_n(A)$ surjects onto $EL_n(A/I)$ for any ring $A$ and ideal $I$).
\vskip .1cm
$\mbox{ (D1) }$ Note that $EL_n(R)$ is generated by
$\Sigma_{12}\cup \Sigma_{23}\cup\Sigma_{31}$.
Let $D1$ consist of relations that express the elements of
$\Sigma\setminus(\Sigma_{12}\cup \Sigma_{23}\cup\Sigma_{31})$
in terms of the elements of $\Sigma_{12}\cup \Sigma_{23}\cup\Sigma_{31}$.

$\mbox{ (D2) }$ The groups $\la \Sigma_{12},\Sigma_{23}\ra$, $\la
\Sigma_{23},\Sigma_{31}\ra$ and $\la
\Sigma_{31},\Sigma_{12}\ra$ are finitely generated
nilpotent groups of class $2$. Thus, they are finitely presented.
Let $D2$ be the union of sets of defining relations for these three groups.

$\mbox{ (D3) }$ Assume that $R_0=\Z$. Then the subgroup $$\la
\Sigma_{11}(0),\Sigma_{22}(0),\Sigma_{12}\ra$$ is equal to
$$\left \{ \left(\begin{array}{ccc} A& B& 0\\ 0 & C&0\\0&0&I_c\end{array}\right
)\in \M_{n\times n}(R): A\in SL_{a}(\dbZ), C\in SL_b(\dbZ), B\in
\M_{a\times b}(\sum_{k=0}^d \dbZ
 x_k)\right\},$$
and so it is finitely presented. The same is true for the groups
 $$\la \Sigma_{22}(0),\Sigma_{33}(0),\Sigma_{23}\ra\textrm{\ and \ }\la \Sigma_{11}(0),\Sigma_{33}(0),\Sigma_{31}\ra.$$
Let $D3$ be the union of sets of defining relations for these three groups.
\vskip .15cm
Let $\widetilde\Sigma$ be a ``copy'' of the set $\Sigma$
(elements of $\widetilde\Sigma$ and $\Sigma$ will be denoted
by the same symbols, but $\widetilde\Sigma$ will not be thought
of as a subset of $EL_n(R)$). Let $\widetilde\Sigma_{ij}(m)$ (resp. $\widetilde\Sigma_{ij}$)
be the subset of $\widetilde\Sigma$ naturally corresponding to
$\Sigma_{ij}(m)$ (resp. $\Sigma_{ij}$).
Define the group $\Gamma$ by setting
\begin{align*}
&\qquad\Gamma=\la \widetilde\Sigma \mid D1\cup D2\ra &&\mbox{ if } R_0 \mbox{ is finite}\qquad&&&\\
&\qquad\Gamma=\la \widetilde\Sigma  \mid D1\cup D2\cup D3\ra &&\mbox{ if } R_0=\dbZ\qquad&&&
\end{align*}
It is clear that $\Gamma$ is finitely presented and surjects onto
$EL_n(R)$. Let $\pi:\Gamma\to EL_n(R)$ be the canonical
surjection.

\vskip .15cm

{\bf Proof of property $(T)$ for $\Gamma$.}
We shall now prove that $\Gamma$ has property $(T)$ if either
$n\geq 7$ and $R_0=\dbZ$, or $n\geq 3$ and $R_0$ is finite, with
$|R_0|\geq 5$. We shall also  estimate the Kazhdan constant
$\kappa(\Gamma,\widetilde\Sigma)$.

For each $m\in [0,d]$ consider the following subgroups of $\Gamma$:
$$\Gamma_1(m) =\la \widetilde\Sigma_{12}(m)\ra,\quad
\Gamma_2(m) =\la \widetilde\Sigma_{23}(m)\ra,\quad
\Gamma_3(m) =\la \widetilde\Sigma_{31}(m)\ra.$$
 It is clear that
\begin{gather*}
\pi(\Gamma_1(m))= \left \{\left (\small\begin{array}{ccc} I_a & D & 0\\
0 & I_b & 0\\
0 & 0 & I_c \end{array}\right )\in  \M_{n\times n}(R)\ : \ D\in
\M_{a\times b} (R_0 x_m)\right\}\\
\pi(\Gamma_2(m))= \left \{\left (\small\begin{array}{ccc} I_a & 0 & 0\\
0 & I_b & D\\
0 & 0 & I_c \end{array}\right )\in  \M_{n\times n}(R)\ : \ D\in
\M_{b\times c} (R_0 x_m)\right\}\\
\pi(\Gamma_3(m))= \left \{\left (\small\begin{array}{ccc} I_a & 0 & 0\\
0 & I_b & 0\\
D & 0 & I_c \end{array}\right )\in  \M_{n\times n}(R)\ : \ D\in
\M_{c\times a} (R_0 x_m)\right\}
\end{gather*}
For each $i=1,2,3$, we set
$$S_i=\bigcup\limits_{m=0}^d \Gamma_i(m) \mbox{ and } \Gamma_i=\la S_i\ra.$$
By relations (D2), for $i=1,2,3$ the group $\Gamma_i$ is isomorphic to
the direct product $\Gamma_i(0)\times\ldots \times \Gamma_i(d)$, whence
the subgroups $\{\Gamma_i(j): j\in [0,d]\}$ of $\Gamma_i$ are pairwise
$0$-orthogonal. Thus, Corollary~\ref{Tn} implies that
\begin{equation}
\label{Kazhdan_easy} \kappa(\Gamma_i,S_i)\geq \sqrt{\frac{2}{d+1}}
\end{equation}
\vskip .1cm
By relations $(D1)$, the group $\Gamma$ is generated by $\Gamma_1,\Gamma_2$ and
$\Gamma_3$. Next we compute orthogonality constants between these subgroups.
\vskip .1cm
By $q$ we will denote the minimal index of a proper ideal of $R_0$.
Thus, \vskip .1cm \centerline{$q=|R_0|$ if $R_0$ is a finite field
and $q=2$ if $R_0=\dbZ$.}
\begin{Claim}
\label{claim:3subgps}
Let $\eps_1=\frac{1}{\sqrt{q^c}}$,
$\eps_2=\frac{1}{\sqrt{q^a}}$ and $\eps_3=\frac{1}{\sqrt{q^b}}$.
Then $\Gamma_1$ and $\Gamma_2$ are $\eps_3$-orthogonal, $\Gamma_2$ and $\Gamma_3$ are
$\eps_1$-orthogonal and $\Gamma_3$ and $\Gamma_1$ are $\eps_2$-orthogonal.
\end{Claim}
\begin{proof} We shall only prove that $\Gamma_1$ and $\Gamma_2$ are $\eps_3$-orthogonal;
proofs of the other two statements are analogous.
Relations $(D2)$ ensure that the group $\Gamma_{1,2}=\la\Gamma_1,\Gamma_2\ra$
maps injectively to $EL_n(R)$. In particular, $\Gamma_{1,2}$ is nilpotent of
class two, and we can identify $\Gamma_1$ and $\Gamma_2$ as sets with
$\M_{a\times b}(\sum_{k=0}^d R_0 x_k)$ and $\M_{b\times c}(\sum_{k=0}^d R_0 x_k)$, respectively.
Furthermore, $\Gamma_{1,2}$ becomes a Noetherian $A$-group with $X=\Gamma_1$, $Y=\Gamma_2$
and $A=\M_{b\times b}(R_0)$, where $A$ acts on $X$ (resp. $Y$) by right (resp. left) multiplication.
The smallest size quotient module of
$A$ is $\dbF_q ^b$ (where $\dbF_q$ is a field with $q$ elements).
Thus, Claim~\ref{claim:3subgps} follows from
Corollary~\ref{cor:main}.
\end{proof}
Note that $a,b\geq 2$ and $c\geq 3$ whenever $n\geq 7$. If
$\eps_1,\eps_2$, $\eps_3$ are as in the statement of
Claim~\ref{claim:3subgps}, then
$\frac{\sqrt{2}\max\{\eps_1,\eps_2\}}{\sqrt{1-\eps_3}}<1$ whenever
$n\geq 7$ or $q\geq 5$. Thus, Corollary~\ref{T3} implies that the
Kazhdan constant for $\Gamma=\la \Gamma_1,\Gamma_2,\Gamma_3\ra$
with respect to $S_1\cup S_2\cup S_3$ is positive. In fact, with
the exception of the cases $n=7,8$, $R_0=\dbZ$, we have
$\max\{\eps_1,\eps_2,\eps_3\}<1/2$, so we can use
Corollary~\ref{Tn} instead of Corollary~\ref{T3}, which yields a
better estimate for the Kazhdan constant. A straightforward computation
yields the following lower bound for $\kappa(\Gamma,S_1\cup S_2\cup S_3)$:

\begin{Corollary}
\label{cor2} Let $S=S_1\cup S_2\cup S_3$.
Assume that $n\geq 7$ or
$q\geq 5$. Then
$$\kappa(\Gamma,S)\geq \frac{C_{n,q}}{\sqrt{d+1}}$$
where $C_{n,q}=\frac{1}{6}$ if $n=7$ and $R_0=\dbZ$,
$\,\,\,C_{n,q}=\frac{1}{4}$ if $n=8$ and $R_0=\dbZ$, and
$C_{n,q}=\sqrt{\frac{2}{3}(1-2(\frac{1}{q})^{[n/3]})}$ in all
other cases.
\end{Corollary}

If $R_0$ is a finite field, the set $S$ is finite, so
Corollary~\ref{cor2} implies that $\Gamma$ has property $(T)$
(though some work still has to be done to compute the Kazhdan
constant with respect to $\widetilde\Sigma$). In order to finish
the proof of property $(T)$ in the case $R_0=\dbZ$ we use the
following result of Kassabov~\cite[Corollary~5.6]{Ka1}:

\begin{Proposition}[Kassabov]
Let $i,j\geq 2$, and let $$H_{i,j}=(SL_i(\dbZ)\times
SL_j(\dbZ))\ltimes M_{i\times j}(\dbZ)$$ where $SL_i(\dbZ)$ acts
by left multiplication and $SL_j(\dbZ)$ by right multiplication.
Let $T_{i,j}$ be the generating set of $H_{i,j}$ consisting of
the union of the sets of elementary matrices with $1$ off the diagonal
in $SL_i(\dbZ)$ and $SL_j(\dbZ)$ and the set of
$ij$ matrices in $M_{i\times j}(\dbZ)$ with $1$ at one position
and $0$ everywhere else.

Let $V$ be a unitary representation of $H_{i,j}$, and let $v\in V$
be an $(T_{i,j},\eps)$-invariant vector (for some $\eps>0$).
Then for any $g\in M_{i\times j}(\dbZ)$ we have
$$\|gv-v\| < \alpha (i+j)\cdot \eps \| v\|$$
where $\alpha:\dbN\to\dbR$ is the function defined by
$\alpha(s)=\sqrt{10s+120}+12$. \label{Kassabov1}
\end{Proposition}

Defining relations (D3) for $\Gamma$ ensure that for any $m\in[0,d]$
there is a natural embedding $\iota_{m}: H_{a,b}\to \Gamma$
such that $\iota_{m}(M_{a\times b}(\dbZ))=\Gamma_1(m)$ and
$\iota_{m}(T_{a,b})= \widetilde\Sigma_{11}(0)\cup \widetilde\Sigma_{22}(0)\cup \widetilde \Sigma_{12}(m)$.
Then by Proposition~\ref{Kassabov1}, if $V$ is any unitary representation
of $\Gamma$ and $v\in V$ is $(\widetilde\Sigma,\eps)$-invariant, then
\begin{equation}
\label{AAA}
\|gv-v\| < \alpha ([2n/3]+1)\cdot \eps \| v\|
\end{equation}
for any $g\in S_{1}$ (recall that $S_1=\Gamma_1(0)\cup
\Gamma_1(1)\cup\ldots\cup \Gamma_1(d)$). Similarly, \eqref{AAA} holds
for any $g\in S_2$ and $g\in S_3$. It follows that
$$\kappa(\Gamma,\widetilde\Sigma)\geq
\frac{\kappa(\Gamma,S)}{\sqrt{20n/3+130}+12}>0.$$ Since $\widetilde\Sigma$
is finite, we conclude that $\Gamma$ has property $(T)$, and by
Corollary~\ref{cor2}, the Kazhdan constant $\kappa(\Gamma,\widetilde\Sigma)$
can be estimated as follows:
\begin{Proposition}
\label{Kazhdan_ELn}
Assume that $R_0=\dbZ$ and $n\geq 7$. Then
\begin{equation}
\kappa(\Gamma,\widetilde\Sigma)\geq
\frac{C_n}{\sqrt{d+1}(\sqrt{20n/3+130}+12)} \mbox{ where }
\end{equation}
$C_{n}=\frac{1}{6}$ if $n=7$, $\,\,\,C_{n}=\frac{1}{4}$ if $n=8$,
and $C_{n}=\sqrt{\frac{2}{3}(1-(\frac{1}{2})^{[n/3]-1})}$ for
$n\geq 9$.
\end{Proposition}
Finally, in the case when $R_0$ is a finite field, using bounded
generation we obtain the following bound for
$\kappa(\Gamma,\widetilde\Sigma)$:
\begin{Proposition}
\label{Kazhdan_ELn2} Assume that $R_0$ is a field and
$|R_0|=p^s\geq 5$ (with $p$ prime). Then
\begin{equation}
\kappa(\Gamma,\widetilde\Sigma)\geq
\frac{\sqrt{\frac{2}{3}(1-2(\frac{1}{p^s})^{[n/3]})}}{\sqrt{d+1}\cdot
([\frac{n+2}{3}])^2\cdot p s}
\end{equation}
\end{Proposition}
{\bf An explicit presentation for a finitely presented cover of $EL_n(R)$.}
We shall finish this subsection by defining a finitely presented cover $\Delta$ of
$EL_n(R)$ by an explicit set of relations. If $R_0$ is $\dbZ$ or a prime field,
the group $\Delta$ will be a quotient of the group $\Gamma$ constructed above.
It is possible to write down a presentation for $\Gamma$
itself, but such presentation would look cumbersome because
the definition of $\Gamma$ is not completely canonical.

For the convenience of the reader we recall all relevant notations
in the statement of the following theorem.
\begin{Theorem}
\label{explicitpres}
Let $R_0=\dbZ$ or a finite field. Choose a basis
$\{\alpha_1=1,\ldots,\alpha_s\}$ for $R_0$ over $\dbZ$.
Let $\{c_{tt'}^u\in \dbZ : 1\leq t,t',u\leq s\}$ be such that
$\alpha_t\alpha_{t'}=\sum\limits_{u=1}^s c_{tt'}^u \alpha_u.$
Let $R$ be an associative $R_0$-algebra, generated over $R_0$ by the set $\{x_1,\ldots,x_d\}$,
and let $x_0=1$.
Let $n\geq 3$ be an integer. Assume in addition that $n\geq 7$
or  $R_0$ is finite, with $|R_0|\geq 5$.
Let $\Delta$ be the group generated by the set
$$\widehat\Sigma=\{e_{ij}(\alpha_t x_m) : i,j\in \{1,\ldots,n\}, i\neq j,\,\, 0\leq m\leq d,\, 1\leq t\leq s \}$$
subject to the following relations:
\begin{align*}
&\mbox{ (E0) }& &e_{ij}(\alpha_t x_m)^p=1 \mbox{ if }R_0\mbox{ is a field of characteristic } p &\\
&\mbox{ (E1) }& &[e_{ij}(\alpha_t x_m), e_{i'j'}(\alpha_{t'} x_{m'})] = 1 \mbox{ if } \{i,i'\}\cap\{j,j'\}=\emptyset&\\
&\mbox{ (E2) }& &[e_{ij}(\alpha_{t} x_m), e_{jk}(\alpha_{t'}x_0)] =
\prod\limits_{u=1}^s e_{ik}(\alpha_{u} x_m)^{c_{tt'}^u} \mbox{ if } i,j,k \mbox{ are distinct} &\\
&\mbox{ (E3) }& &[e_{ij}(\alpha_{t'} x_0), e_{jk}(\alpha_{t} x_m)] =
\prod\limits_{u=1}^s e_{ik}(\alpha_{u} x_m)^{c_{tt'}^u} \mbox{ if } \mbox{ if } i,j,k \mbox{ are distinct}&\\
&\mbox{ (E4) }& &[[e_{ij}(\alpha_{t}x_m), e_{jk}(\alpha_{t'} x_{m'})],e_{i'k'}(\alpha_{t''} x_{m''})]=1
\mbox{ if } \{i,i'\}\cap\{k,k'\}=\emptyset \mbox{ and } j\not\in\{i,k\}&\\
&\mbox{ (E5) }& &[e_{ij}(\alpha_t x_m), e_{jk}(\alpha_{t' }x_{m'})]=[e_{ij'}(\alpha_{t} x_m), e_{j'k}(\alpha_{t'} x_{m'})]
\mbox{ if } i\neq k \mbox{ and } j,j'\not\in\{i,k\}&\\
&\mbox{ (E6) }& &(e_{12}(x_0)e_{21}(x_0)^{-1}e_{12}(x_0))^4=1 \mbox{ if }R_0=\dbZ&
\end{align*}
Then $\Delta$ is a finitely presented group with property $(T)$ which surjects onto $EL_n(R)$.
Furthermore, if $R_0=\dbZ$ (resp. $R_0$ is a field),
the Kazhdan constant $\kappa(\Delta,\widehat\Sigma)$ satisfies
the same inequality as $\kappa(\Gamma,\widetilde\Sigma)$
in the statement of Proposition~\ref{Kazhdan_ELn} (resp. Proposition~\ref{Kazhdan_ELn2}).
\end{Theorem}
\begin{proof} First of all, it is clear that $\Delta$ surjects onto $EL_n(R)$.
As noted before, we may assume that $R$ is the free associative algebra
$R_0\la x_1,\ldots, x_d\ra$, and let $\Gamma$ be the group defined earlier in this subsection.
If $R_0$ is $\dbZ$ or a prime field, we will show that for a suitable choice of relations
$(D1)$ in the definition of $\Gamma$, the group $\Delta$ is a quotient of $\Gamma$, which
will finish the proof. If $R_0$ is a non-prime field, an additional remark will be needed.
For each subset $A$ of $\Sigma$, the corresponding
subset of $\widehat\Sigma$ will be denoted by $\Ahat$.

Relations $(E2)$-$(E3)$ imply that $\Delta$ is generated by the
set
$\widehat\Sigma_{12}\cup\widehat\Sigma_{23}\cup\widehat\Sigma_{31}$.
Thus, we may assume that relations $(D1)$ hold in $\Delta$. More
precisely we take $(D1)$ to be the group words which express the
rest of the elements of $\Sigma$ in terms of $\Sigma_{12}$,
$\Sigma_{23}$ and $\Sigma_{31}$ from $(E2)$ and $(E3)$.

Next we show that in the case $R_0=\dbZ$ relations $(D3)$ hold in $\Delta$.
Relations $(E1)$-$(E3)$ with $m=m'=0$ and $(E6)$ imply that the set $\cup_{i,j=1}^3 \widehat\Sigma_{ij}(0)$
generates a copy of $SL_n(\dbZ)$ inside $\Delta$ (see \cite{Mil}). In particular,
$\la\widehat\Sigma_{11}(0)\ra\cong SL_a(\dbZ)$, $\la\widehat\Sigma_{22}(0)\ra\cong SL_b(\dbZ)$
and $\la\widehat\Sigma_{11}(0)\ra$ commutes with $\la\widehat\Sigma_{22}(0)\ra$.
Relations $(E1)$ ensure that $\la\widehat\Sigma_{12}\ra\cong\M_{a\times b}(\sum_{k=0}^d \dbZ x_k)$.
Finally, relations $(E2)$-$(E3)$ with $m>0$ imply that $\la\widehat\Sigma_{12}\ra$ is normalized by
$\la\widehat\Sigma_{11}(0),\widehat\Sigma_{22}(0)\ra$. Thus, the subgroup
$\la\widehat\Sigma_{11}(0),\widehat\Sigma_{22}(0),\widehat\Sigma_{12}\ra$ maps injectively
to $EL_n(R)$. The same is true for $\la\widehat\Sigma_{22}(0),\widehat\Sigma_{33}(0),\widehat\Sigma_{23}\ra$
and $\la\widehat\Sigma_{33}(0),\widehat\Sigma_{11}(0),\widehat\Sigma_{31}\ra$. Thus,
$\Delta$ satisfies $(D3)$.

Finally, consider relations $(D2)$. Relations
$(E1), (E4)$ and $(E5)$ are easily seen to imply that the subgroup
$\la \widehat\Sigma_{12},\widehat\Sigma_{23}\ra$
is nilpotent of class two, and relations $(E2)$-$(E3)$ imply
that $\la \widehat\Sigma_{12},\widehat\Sigma_{23}\ra$
is a Noetherian $A$-group for $A=\M_{b\times b}(R_0)$. Similar results
hold for $\la \widehat\Sigma_{23},\widehat\Sigma_{31}\ra$
and $\la \widehat\Sigma_{31},\widehat\Sigma_{12}\ra$.
If $R_0$ is $\dbZ$ or a prime field, it follows easily that all relations
$(D2)$ hold in $\Delta$. The latter is not true if $R_0$ is a non-prime field,
but this does not change the argument. Indeed, the only place
in the proof of property $(T)$ for $\Gamma$ where relations $(D2)$ were
used was Claim~\ref{claim:3subgps}, and the relations in $\Delta$ established above
clearly suffice for the proof of that claim to work.
\end{proof}
\begin{Remark} Arguing as in Appendix~A, one can show that relation (E6) can be
omitted. Thus, the proof presented in this subsection yields a finitely presented group
with property $(T)$ which is a cover not only for $EL_n(R)$, but also for $St_n(R)$
(assuming
 $n\ge 7$ or $R$ is an algebra over $\F_q$ with $|q|\ge 5$).
\end{Remark}
\subsection{Some concluding remarks}
As a natural extension of Theorem~\ref{thm:main}, it would be interesting to
determine whether the analogues of the groups $EL_n(R)$ and $St_n(R)$ corresponding to
other root systems have property $(T)$.
For any finite root system $\Phi$ and a commutative ring $R$, one can define the associated simply-connected Chevalley group $G_{\Phi}(R)$ and the
Steinberg group $St_{\Phi}(R)$ which maps onto $G_{\Phi}(R)$.\footnote
{There are also standard ways to define $G_{\Phi}(R)$ when $\Phi$ is of type $B_n$, $C_n$
or $D_n$ and $R$ is a (possibly non-commutative) ring with involution.}
 The groups $EL_n(R)$ and $St_n(R)$
correspond to the root system $A_{n-1}$.

Probably the simplest case (excluding type $A$ systems) is when
the Dynkin diagram of $\Phi$ is simply laced (that is, $\Phi$ is
of type $D$ or $E$), because in this case any rank two subsystem
of $\Phi$ is of type $A$. If in addition $R$ is an algebra over a
finite field $F$, with $|F|>\!>{\rm rank}(\Phi)$, then
$St_{\Phi}(R)$ has property $(T)$ since this group can be realized
as the quotient of a suitable Kac-Moody-Steinberg group satisfying
the assumptions of Corollary~\ref{corKMS} (see next section).
However, we do not know whether the proof presented in this
section can be adapted since we do not know whether $St_{\Phi}(R)$
is an $A_2$-group or has a similar structure. The treatment of
root systems with non-simply-laced Dynkin diagrams would almost
certainly require the analogues of the results of Section 4 for
groups of nilpotency class $3$ and $4$ satisfying some additional
conditions.

\vskip .1cm
We finish this section with an interesting application of
Theorem~\ref{thm:main}. It is well-known that a discrete group
which is amenable and has property $(T)$ must be finite. In an
attempt to generalize this fact, Lubotzky and Weiss proposed the
following conjecture (see \cite[Conjecture 1.2]{LW}):
\begin{Conjecture}[Lubotzky-Weiss]
\label{conj:LW}
Let $K$ be an infinite compact group. Then
$K$ cannot contain finitely generated dense subgroups
$A$ and $B$ where $A$ is amenable and $B$ has property $(T)$.
\end{Conjecture}
As one of the examples supporting this conjecture, Lubotzky and Weiss considered the
profinite group $G_p=\prod_{n\geq 2} SL_n(\Fp)$ for a fixed prime $p$, showed
that $G_p$ contains a finitely generated dense amenable subgroup and argued
that all known (at the time) discrete groups with property $(T)$ cannot
be densely embedded in $G_p$. However, in \cite{Ka2}, Kassabov proved that
a very similar group $G'_p=\prod_{n\geq 2} SL_{3n}(\Fp)$ ``almost'' provides
a counterexample to this conjecture: it contains finitely generated dense subgroups $A$ and $B$
where $A$ is amenable and $B$ has property $(\tau)$. The existence of a dense
amenable subgroup in $G'_p$ follows from \cite{LW} since $G'_p$ is a quotient of $G_p$.
On the hand, Kassabov observed that $G'_p\cong \prod_{n\geq 2} EL_{3}(M_n(\Fp))\cong EL_3\Big(\prod_{n\geq 2}M_n(\Fp)\Big)$.
The profinite ring $\prod_{n\geq 2}M_n(\Fp)$ is well-known to be (topologically) finitely generated, and
therefore $G'_p$ contains a dense subgroup of the form $EL_3(S)$ where $S$ is some finitely
generated (discrete) ring with $1$. In \cite{Ka2}, Kassabov was able to show that $EL_3(S)$
has property $(\tau)$. Having Theorem~\ref{thm:main} at our disposal, we now know that
$EL_3(S)$ has property $(T)$, and thus the group $G'_p$ is indeed a counterexample
to Conjecture~\ref{conj:LW}. We note that Kassabov\footnote{private communication}
constructed a different example to Conjecture~\ref{conj:LW} using \cite{Sh3}.

\section{Property $(T)$ for Kac-Moody-like groups}

In this section we introduce a large class of groups which we call
{\it Kac-Moody-Steinberg }(KMS) groups and show that many of these groups
have property $(T)$. Homomorphic images of groups in this class
include many Steinberg groups as well as certain parabolic subgroups of ordinary
Kac-Moody groups, which justifies the proposed name. The relationship between
KMS groups and Kac-Moody groups (explained below in more detail) is not used in our proofs
at all; instead, it yields an alternative proof of property $(T)$ for
Kac-Moody groups.
\vskip .1cm
\subsection{Basic Kac-Moody-Steinberg groups.}
Let $R$ be an associative ring with $1$, and let $X$ be a finite graph
without loops or multiple edges. We denote the vertices of $X$ by integers
$\{1,2\,\ldots,d\}$.

For each $i\in [1,d]$, let $G_i$ be the group with elements
$\{x_i(r): r\in R\}$ subject to relations $x_i(r) x_i(s)=x_i(r+s)$
for $r,s\in R$. Thus, each $G_i$ is isomorphic to $(R,+)$.
Let $G(X,R)$ be the group generated by $G_1,\ldots, G_d$
subject to the following relations:
\begin{itemize}
\item If $i,j\in [1,d]$ and $(i,j)\not \in E(X)$, then $G_i$ and $G_j$ commute.

\item If $i,j\in [1,d]$ and $(i,j)\in E(X)$,
then $[x_i(r), x_j(s)]=[x_i(1), x_j(rs)]$ for any $r,s\in R$,
and $[G_i,G_j]$ commutes with both $G_i$ and $G_j$.
\end{itemize}
The group $G(X,R)$ will be called the {\it basic Kac-Moody-Steinberg (KMS) } group corresponding
to the graph $X$ and the ring $R$. It is easy to see that
$G(X,R)$ is finitely presented whenever $(R,+)$ is finitely generated.
Two special cases are worth mentioning.

If $X$ is chain of length $d$, that is, $E(X)=\{(1,2),(2,3),\ldots, (d-1,d)\}$,
then $G(X,R)$ surjects onto the upper-unitriangular subgroup of $EL_{d+1}(R)$
via the map $x_i(r)\mapsto e_{i,i+1}(r)$.

If $X$ is a cycle of length $d$, that is, $E(X)=\{(1,2),\ldots, (d-1,d), (d,1)\}$,
then $G(X,R)$ naturally surjects onto the Steinberg group $St_d(R)$
(and hence also onto $EL_d(R)$) via the map $x_i(r)\mapsto E_{i,i+1}(r)$,
where indices are taken mod $d$. There is another natural mapping
$\pi:G(X,R)\to St_d(R[t])$ (where $R[t]$ are polynomials over $R$, and $t$ commutes with $R$).
It is defined by
$$\pi(x_i(r))= E_{i,i+1}(r)\mbox { for }1\leq i\leq d-1
\mbox{ and }\pi(x_d(r))=e_{d,1}(rt).$$
If $R$ is commutative, the projection of $\pi(G(X,R))$ to $EL_d(R[t])$
is the subgroup of matrices in $EL_d(R[t])$
which are upper-unitriangular mod $t$. This group is in fact the `positive unipotent'
subgroup of the affine Kac-Moody group of type $\widehat{A_d}$ over $R$.

More generally, for any graph $X$, the basic KMS group $G(X,R)$ surjects onto
the `positive unipotent' subgroup of the Kac-Moody group over $R$ whose
associated Dynkin diagram is equal to $X$.
\subsection{``Mixed'' KMS groups.} Once again, let $X$ be a finite graph
with vertices $\{1,2,\ldots,d\}$, and let $M_1,\ldots, M_d$ be a collection of abelian
groups. Suppose that for every edge $(i,j)\in E(X)$, with $i<j$, there exists
a ring $R_{i,j}$ such that $M_i$ is a right $R_{i,j}$-module and $M_j$ is a left $R_{i,j}$-module.

For each $i\in [1,d]$, let $G_i$ be the group with elements $\{x_i (a): a\in M_i\}$
subject to relations $x_i(a)x_i(a')=x_i(a+a')$ for any $a,a'\in M_i$, so $G_i\cong M_i$.
Let $G=G(X,\{M_i\},\{R_{i,j}\})$ be the group generated by $G_1,\ldots, G_d$
subject to the following relations:

\begin{itemize}
\item If $i,j\in [1,d]$ and $(i,j)\not\in E(X)$, then $G_i$ and $G_j$ commute

\item If $i<j\in [1,d]$ and $(i,j)\in E(X)$,
then $[x_i(ar), x_j(b)]=[x_i(a), x_j(rb)]$ for any $a\in M_i$, $b\in M_j$ and $r\in R_{i,j}$,
and $[G_i,G_j]$ commutes with both $G_i$ and $G_j$.
\end{itemize}

The group $G(X,\{M_i\},\{R_{i,j}\})$ will be called the {\it mixed Kac-Moody-Steinberg (KMS) }
group corresponding to the triple $(X,\{M_i\},\{R_{i,j}\})$. As in the case of basic KMS groups,
$G(X,\{M_i\},\{R_{i,j}\})$ is finitely presented whenever each $M_i$ is finitely generated
(as a group). If $R$ is a ring with $1$ and we set $R_{i,j}=R$ for each $(i,j)\in E(X)$
and $M_i=(R,+)$ for each $i$, then the mixed KMS group $G(X,\{M_i\},\{R_{i,j}\})$
coincides with the basic KMS group $G(X,R)$.
\vskip .12cm
Let $X,\{M_i\},\{R_{i,j}\}$ be as above and $G=G(X,\{M_i\},\{R_{i,j}\})$.
Assume that each $M_i$ is finitely generated, and let $\{G_i\}_{i=1}^d$ be defined as above.
The following result is a direct consequence of Corollary~\ref{cor:main}:
\begin{Proposition}
Let $(i,j)\in E(X)$. Then $G_i$ and $G_j$ are $\frac{1}{\sqrt{m_{i,j}}}$-orthogonal,
where $m_{i,j}$ is the minimal index of a proper right ideal in $R_{i,j}$.
\end{Proposition}

\begin{Corollary}
\label{corKMS}
Let $\{m_{i,j}\}$ be as in the previous proposition, $m=\min\{m_{i,j}\}$,
and assume that $m > (d-1)^2$.
Then $\kappa(G,\cup G_i)\geq \sqrt{\frac{2}{d}(1-\frac{d-1}{\sqrt{m}})}$.
In particular, if each $M_i$ is finite, then $G$ has $(T)$.
\end{Corollary}
\begin{proof} This is a direct consequence of Corollary~\ref{Tn}.
\end{proof}

\begin{Corollary}
\label{corKM} Let $A$ be a $d\times d$ generalized Cartan matrix
with $0$ or $-1$ off the diagonal, let $F$ be a finite field, and
let $G_{KM}(A,F)$ be the corresponding simply-connected Kac-Moody
group. Let $U=U(A,F)$ be the ``positive unipotent'' subgroup of
$G_{KM}(A,F)$, that is, the subgroup of $G_{KM}(A,F)$ generated by positive
root subgroups. If $|F|>(d-1)^2$, then $U$ has property $(T)$, and
$\kappa(U,S)\geq \sqrt{\frac{2}{d}(1-\frac{d-1}{\sqrt{|F|}})}$,
where $S$ is the union of simple root subgroups.
\end{Corollary}
\begin{proof} It is clear from the definition of Kac-Moody groups that $U(A,F)$ is
a quotient of the basic KMS group $G(Dyn(A),F)$ where $Dyn(A)$ is the Dynkin
diagram of $A$, and thus we are done by Corollary~\ref{corKMS}.
\end{proof}
\begin{Remark}
The work of Dymara and Januszkiewicz~\cite{DJ} implies that the group $U(A,F)$ has property $(T)$
whenever $A$ is a $d\times d$ $2$-spherical matrix (that is, $a_{ij}a_{ji}\leq 3$
for any $i\neq j$) and $|F|>\frac{1}{25}1764^{d-1}$, but does not yield explicit Kazhdan constants.
\end{Remark}

\begin{Remark} By Corollary~\ref{corKMS}, the basic KMS group $G(X,\Fq)$
has property $(T)$ for $q>(|X|-1)^2$. We do not know whether this restriction
on $q$ can be improved, but we know that it cannot be completely eliminated.
A computer calculation with GAP showed that if $X$ is a complete graph
on 3 vertices, then $G(X,\F_2)$ has a subgroup of finite
index with infinite abelianization, and so $G(X,\F_2)$ does not
have property $(T)$. We want to thank  Benjamin Klopsch for performing
this calculation.
\end{Remark}

\subsection{Golod-Shafarevich groups with property $(T)$.}
In \cite{Er}, the work of Dymara and Januszkiewicz was used to
produce the first examples of Golod-Shafarevich groups with
property $(T)$. In this subsection we generalize and improve the
main result of \cite{Er}. Unlike the latter paper, which dealt
with Kac-Moody groups, we will work with Kac-Moody-Steinberg
groups, so verification of Golod-Shafarevich inequality will be
straightforward.

We briefly recall the definition of Golod-Shafarevich groups.
For more details, unexplained terminology and motivation
the reader is referred to \cite{Er} and references therein.
\begin{Definition}\rm
Let $\la X | R\ra$ be a group presentation with $|X|<\infty$.
Given a prime $p$, let $r_i$ be the number of relations in $R$
which have degree $i$ with respect to the Zassenhaus $p$-filtration.
The presentation $\la X | R\ra$ is said to satisfy {\it the Golod-Shafarevich
condition with respect to $p$} if there exists a real number
$0<t<1$ such that $1-|X|t+\sum_{i=1}^{\infty}r_i t^i<0$.
\end{Definition}
\begin{Definition}\rm A group $G$ is called a {\it Golod-Shafarevich
group with respect to $p$} if it has a presentation satisfying
the Golod-Shafarevich condition with respect to $p$.
\end{Definition}

\begin{Proposition} Let $d\geq 6$, and let $p>(d-1)^2$ be a prime.
Let $K_d$ be the complete graph on $d$ vertices, and let
$G$ be the basic KMS group $G(K_d,\Fp)$. Then $G$ is a
Golod-Shafarevich group with respect to $p$ and has property $(T)$.
\label{GS1}
\end{Proposition}
\begin{proof}
We already established the property $(T)$ part, so we only need
to verify the Golod-Shafarevich condition for $G$.
By definition, the group $G$ is given by the following presentation:
$$G=\la x_1,\ldots, x_d\mid x_i^p=1, [x_i,x_j,x_i]=1 \mbox{ for any }1\leq i\neq j\leq d \ra.$$
The Hilbert $p$-series of this presentation is
$H(t)=1-dt+d(d-1)t^3+dt^p$. An easy computation shows that $H(\frac{1}{\sqrt{3(d-1)}})<0$
whenever $p\geq 5$ and $d\geq 6$, so $G$ is Golod-Shafarevich (with respect to $p$).
\end{proof}

Proposition~\ref{GS1} improves the result of \cite{Er} only quantitatively
(that is, it holds under milder restrictions on $p$). The main thing
that is unsatisfactory about either statement is that for a fixed prime $p$
it does not allow one to construct a Golod-Shafarevich group with $(T)$
and with arbitrarily large number of generators -- it is easy to see
that the minimal number of generators for $G(K_d, F_p)$ is equal to $d$.
This problem can now be resolved using mixed KMS groups.

\begin{Proposition}
\label{GS2}
Let $n\geq 99$ be an integer and $p> 64$ be a prime. Then there exists
a group $G$ with property $(T)$ such that $G$ is Golod-Shafarevich with
respect to $p$ and $d_p(G)=n$, where $d_p(G)$ is the minimal number of
generators of the pro-$p$ completion of $G$.
\end{Proposition}
\begin{proof}
Divide $n$ by $9$ with remainder: $n=9s+u$ where $0\leq u<9$. Let
$X=K_9$ be the complete graph on $9$ vertices, and consider
the mixed KMS group $G=G(X,\{R_{ij}\}, \{M_i\}\})$
where $R_{ij}=\Fp$ for any $1\leq i<j\leq 9$, and $M_i=\Fp^{s_i}$
where $s_i=s+1$ for $1\leq i\leq u$ and $s_i=s$ for $i>u$.
Note that $\sum_{i=1}^9 s_i=9s+u=n$

Then $G$ is given by the following presentation:
\begin{multline*}
\la x_{i,k}, 1\leq i\leq 9, 1\leq k\leq s_i\mid x_{i,k}^p=1,
[x_{i,k},x_{i,l}]=1, \quad [x_{i,k},x_{j,l},x_{i,m}]=1.
\ra
\end{multline*}
(note that for each $i$, the copy of $M_i$ in $G$
is the subgroup $\la\{x_{i,1},\ldots, x_{i,s_i}\}\ra$.

The Hilbert $p$-series of the above presentation is
$$H(t)=1-nt+\sum_{i=1}^9 {s_i\choose 2} t^2 + \sum_{1\leq i\neq j\leq 9}(s_i^2 s_j) t^3 + nt^p.$$
A not so pleasant but straightforward computation shows that
$H(\frac{1}{s\sqrt{24}})<0$ whenever $s\geq 11$ and $p\geq 5$, so
$G$ is Golod-Shafarevich with respect to $p$. It is clear from
the presentation that $d_p(G)=n$. Finally, $G$ has property $(T)$
by Corollary~\ref{corKMS}.
\end{proof}
\section{Appendix A}
In this appendix we explain why the proof of
Proposition~\ref{Kassabov2} in \cite{Ka2} immediately yields
Proposition~\ref{Kassabov_St}.
The key result in \cite{Ka2} on
which Proposition~\ref{Kassabov2} depends is computation of the
relative Kazhdan constant for the pair $((EL_{p}(R)\times
EL_{q}(R))\ltimes M_{pq}(R), M_{pq}(R))$ (see
Proposition~\ref{Kassabov4} below). Note that this result is a
generalization of Proposition~\ref{Kassabov1} (except for a weaker
Kazhdan constant).

As before, let $\{x_0=1,x_1,\ldots,x_d\}$ be a generating set for $R$.
Let $$H_{p,q}=((EL_{p}(R)\times EL_{q}(R))\ltimes M_{p\times q}(R)$$ where
$M_{p\times q}(R)$ denotes $p\times q$ matrices over $R$, the group
$EL_p(R)$ acts on $M_{p\times q}(R)$ by left multiplication and $EL_{q}(R)$
by right multiplication. We do not assume that $p\geq 2$ and $q\geq 2$
(we set $EL_1(R)$ to be the trivial group).
Let $T_{p,q}$ be the subset of $H_{p,q}$ consisting of
the union of the sets of elementary matrices with one of the $x_i$ off the diagonal
in $EL_{p}(R)$ and $EL_{q}(R)$ and the set of
$pq$ matrices in $M_{p\times q}(R)$ with $1$ at one position
and $0$ everywhere else. In fact, $T_{p,q}$ is a generating set
for $H_{p,q}$ if $p\geq 3$ and $q\geq 3$, but this fact is not essential for the proof.
The following result is a reformulation of \cite[Theorem~1.9]{Ka2}:

\begin{Proposition}[Kassabov]
\label{Kassabov4}
The pair $(H_{p,q}\ltimes M_{p\times q}(R), M_{p\times q}(R))$ has relative property $(T)$.
Furthermore,
$$\kappa(H_{p,q}\ltimes M_{p\times q}(R), M_{p\times q}(R); T_{p,q})\geq \frac{1}{\alpha (d,p+q)}$$
where $\alpha:\dbN\times\dbN\to\dbR$ is the function defined by
$\alpha(d,s)=6\sqrt{2}(\sqrt{d}+3)+\sqrt{3s}$.
\end{Proposition}

The only information about the group $EL_{p}(R)\times EL_{q}(R)$ used in the proof of
Proposition~\ref{Kassabov4} is its action on $M_{p\times q}(R)$.
Thus, if we let $\Hgal_{p,q}$ be any group surjecting onto $EL_{p}(R)\times EL_{q}(R)$,
let $\Hgal_{p,q}\ltimes M_{p\times q}(R)$ be the semidirect product in which $\Hgal_{p,q}$ acts as $EL_{p}(R)\times EL_{q}(R)$
and $\Tgal_{p,q}$ any subset of $\Hgal_{p,q}$, surjecting onto $T_{p,q}$,
then $(\Hgal_{p,q}\ltimes M_{p\times q}(R), M_{p\times q}(R))$ has relative $(T)$, and
$\kappa(\Hgal_{p,q}\ltimes M_{p\times q}(R), M_{p\times q}(R); \Tgal_{p,q})\geq \frac{1}{\alpha (d,p+q)}$ as well.
As explained in Section~3, there is a corresponding bound for the Kazhdan ratio:
$\kappa_r(\Hgal_{p,q}\ltimes M_{p\times q}(R), M_{p\times q}(R); \Tgal_{p,q})\geq \frac{1}{2\alpha (d,p+q)}$.

Now take any $n\geq 3$, and let $\calI_1,\calI_2,\calI_3$ be defined as in Section~6, and
set $p=|\calI_1|$, $q=|\calI_2|+|\calI_3|$ (so that $p+q=n$).
Let $\Hgal_{p,q}=St_{p}(R)\times St_{q}(R)$ and let
$\Tgal_{p,q}$ be the standard lift of $T_{p,q}$ to $\Hgal_{p,q}$.
Let $\iota:\Hgal_{p,q}\ltimes M_{p\times q}(R)\to St_{n}(R)$ be the canonical embedding.
It is clear that $\iota(\Tgal_{p,q})\subset \Sigma^{st}$ where $\Sigma^{st}$ is the generating set for $St_n(R)$ defined by \eqref{eq:sigmast}.
On the other hand, $\iota(M_{p\times q}(R))=\Xgal_{12}\Xgal_{13}$ in the notations of Proposition~\ref{Kassabov_St}.
Thus, $\kappa_r(St_n(R),\Xgal_{12}; \Sigma^{st})\geq
\kappa_r(\Hgal_{p,q}\ltimes M_{p\times q}(R), M_{p\times q}(R); \Tgal_{p,q})\geq \frac{1}{2\alpha (d,n)}$.
Similarly, $\kappa_r(St_n(R),\Xgal_{ij};\Sigma^{st})\geq \frac{1}{2\alpha (d,n)}$
for any $1\leq i\neq j\leq 3$, and therefore
$\kappa_r(St_n(R),\cup\Xgal_{ij};\Sigma^{st})=
\min\limits_{1\leq i\neq j\leq 3}\kappa_r(St_n(R),\Xgal_{ij};\Sigma^{st})=\frac{1}{2\alpha (d,n)}$.
This finishes the proof of Proposition~\ref{Kassabov_St}.

\begin{Remark} In the above argument we referred to the Steinberg groups $St_{p}(R)$ and $St_{q}(R)$
where $p,q$ could be less than $3$. Our definition of $St_{1}(R)$ and $St_2(R)$ is identical
to that of $St_n(R)$ for $n\geq 3$ given in Section~6; thus $St_1(R)$ is a trivial group and
$St_2(R)$ is the free product of two copies of $(R,+)$. Note that other definitions of $St_2(R)$
exist in the literature, but for us $St_2(R)$ plays a purely auxiliary role.
\end{Remark}

\end{document}